\documentclass[a4paper,11pt, english]{article}
\usepackage[affil-it]{authblk}
\usepackage{nostro}
\usepackage[nocompress]{cite}

\makeatletter
\def\@maketitle{%
  \newpage
  \null
  \vskip 2em%
  \begin{center}%
  \let \footnote \thanks
    {\Large\bfseries \@title \par}%
    \vskip 1.5em%
    {\normalsize
      \lineskip .5em%
      \begin{tabular}[t]{c}%
        \@author
      \end{tabular}\par}%
    \vskip 1em%
    {\normalsize \@date}%
  \end{center}%
  \par
  \vskip 1.5em}
\makeatother

\begin{document}

\title{Parallel computation of interval bases for persistence module decomposition 
}

\author{Alessandro De Gregorio  \orcidlink{0000-0001-7577-3655}
\thanks{email: \texttt{alessandro.degregorio@polito.it}}}
\affil{Department of Mathematical Sciences, Politecnico di Torino, Italy}

\author{Marco Guerra \orcidlink{0000-0003-0033-3748}
\thanks{email: \texttt{marco.guerra@univ-grenoble-alpes.fr}}}
\affil{Institut Fourier, Universit\'{e} Grenoble-Alpes CNRS}

\author{Sara Scaramuccia \orcidlink{0000-0003-4010-6507}
\thanks{email: \texttt{sara.scaramuccia@uniroma2.it}}}
\affil{Department of Mathematics, Universit\`a di Roma Tor Vergata, Italy}

\author{Francesco Vaccarino \orcidlink{0000-0002-0610-9168}
\thanks{email: \texttt{francesco.vaccarino@polito.it}}
\affil{Department of Mathematical Sciences, Politecnico di Torino, Italy }}

\providecommand{\keywords}[1]{\\ \textbf{\textit{Key-words:}} #1}
\providecommand{\subclass}[1]{\\ \textbf{\textit{Mathematics Subject Classification (2020):}} #1}

\date{\today}

\maketitle

\begin{abstract}
A persistence module $M$, with coefficients in a field $\F$, is a finite-dimensional linear representation of an equioriented quiver of type $A_n$ or, equivalently, a graded module over the ring of polynomials $\F[x]$. It is well-known that $M$ can be written as the direct sum of indecomposable representations or as the direct sum of cyclic submodules generated by homogeneous elements. An interval basis for $M$ is a set of homogeneous elements of $M$ such that the sum of the cyclic submodules of $M$ generated by them is direct and equal to $M$. We introduce a novel algorithm to compute an interval basis for $M$. Based on a flag of kernels of the structure maps, our algorithm is suitable for parallel or distributed computation and does not rely on a presentation of $M$. This parallel algorithm outperforms the approach via the presentation matrix and Smith normal form. We specialize our parallel approach to persistent homology modules, and we close by applying the proposed algorithm to tracking harmonics via Hodge decomposition. 
\vspace{.6cm}
\keywords{Graded module decomposition \and  Persistent homology \and  Hodge decomposition \and Hodge Laplacian \and Graded Smith Normal Form}
%
\subclass{13P20 \and 
55N31 \and 
62R40   
}
\end{abstract}

\section{Introduction}
\label{sec:intro}
{\em Persistence module} is a modern name for finite-dimensional representations of an equioriented quiver of type $A_n$ that has become popular within the setting of Topological Data Analysis (TDA) and, more specifically, in connection to \textit{persistent homology}, one of the most successful tools in TDA (see \cite{carlsson2009data}).

First, a quiver $Q$ of type $A_n$ is the Hasse diagram of the linearly ordered set $[n]:=(\{1,\dots,n\},\leq)$. This is an oriented simple graph whose vertices are indexed by $[n]$ and whose set of arrows is $\{(i,i+1)\,:\, i=1,\dots, n-1\}$.
\begin{definition}
A {\em persistence module} $\mathcal{M} = \{(M_i, \varphi_i)\}_{i=1}^n$ is a linear representation  of $Q$ with coefficients in a field $\F$. More explicitly, a persistence module is given by the following datum:
\begin{itemize}
	\item a finite-dimensional $\F$-vector space $M_i$, called the $\tth{i}-${\em step}, for each vertex $i$ in $[n]$;
	\item a linear map $\varphi_{i}: M_i \longrightarrow M_{i+1}$, called $\tth{i}-${\em structure map}, for each arrow  $(i, i+ 1) $ in $[n]$.
\end{itemize}
    \label{def:quiver-representation}
\end{definition}
It is a well-known result that persistence modules can be decomposed, uniquely up to isomorphism, into the direct sum of indecomposable modules (see \cite{gabriel1972}) 
\begin{equation} \label{eq:intro_decomposition}
    \mathcal{M} \cong \bigoplus_{m=1}^N I_{[b_m, d_m]}
\end{equation}
where, for all $1\leq b_m\leq d_m\leq n$, $I_{[b_m, d_m]}$ is the persistence module with steps $(I_{[b_m, d_m]})_i=\F$ for all integers $i$ in the closed interval $[b_m, d_m]$ and zero elsewhere; and structure maps the identity for $i\in [b_m, d_m-1]$ and zero elsewhere. 
The modules $I_{[b_m, d_m]}$ are often called \textit{interval modules} and are the indecomposable representations of $A_n$. The decomposition of~\eqref{eq:intro_decomposition} into interval modules was well-known in the quiver representations community since the 70s, and somehow neglected and rediscovered in persistence several years later \cite{oudot2017persistence}.
The multiset made of the intervals $[b_m,d_m]$ is a complete discrete invariant for the isomorphism classes of finite-dimensional linear representations of type $A_n$, see the works of Abeasis, Del Fra, and Kraft \cite{Abe80,AbeJAlg85,abeasis1981geometry}. In particular, they presented in the early 80s the first example of \textit{barcode} that we know, calling it \textit{diagram of boxes} (see, e.g., Section 2 in \cite{abeasis1981geometry}).

A persistence module $\mathcal{M}$ can be associated with a graded $\F[x]$-module $\alpha(\mathcal{M})$ under a well-known equivalence of categories~\cite{corbet2018,Carlsson2005DCG}, in the following way: given $\mathcal{M}$ as above, $\alpha(\mathcal{M})$ is defined as $\bigoplus_{i\in\N}\alpha(M_i):=M_1\oplus M_2\cdots \oplus M_n\oplus M_n\oplus M_n\cdots$. The grading structure is obtained by setting $xv=\varphi_{i}(v)$, for each $i\in [n]$ and $v\in \alpha(M_i)=M_i$ and $xv=v$ for $v\in\alpha(M_j)=M_n$ for $j>n$. 
In this setting, too, there is a well-known decomposition, where cyclic submodules generated by homogeneous elements play the same role as the indecomposable quivers. Consider indeed the cyclic submodule $I(v)$ of $\alpha(\mathcal{M})$, generated by a homogeneous element $v\in M_b$, for some $b$. Then there are two possibilities for $v$: it has torsion, that is, there is $e\in\N$ such that $x^ev=0_{\mathcal{M}}$, so that $I(v)\cong\F[x]/(x^e)$ or $v$ is torsion-free, so that $I(v)\cong \F[x]$. We denote by $I(b,e)$ the submodule $I(v)$ in the torsion case and by $I(b,\infty)$ in the torsion-free (also named by ``free'') case. We call \textit{interval modules} the modules of the type $I(*,*)$ as their germane in the quiver representations setting.

The theorem of decomposition of a graded module over a graded principal ideal (see Theorem 1 in \cite{webb1985}) domain can now be restated as
\begin{equation}
\label{eq:intro_gradecomp}
\alpha( \mathcal{M}) \cong \bigoplus_{m=1}^N I(b_m, e_m).
\end{equation}
Now, $e_m$ can be an integer or the $\infty$ symbol. Exactly as for the quiver representation case, the multiset of intervals $(b_m,e_m)$ occurring in the decomposition is a complete discrete invariant for the isomorphism classes of persistence modules, usually called the \textit{barcode} in TDA.

Strictly related to the above decompositions into interval submodules is the concept of \textit{interval basis}:
a finite set $\{v_1,\dots, v_N\,:\, v_m\in M_{b_m}, \forall m\}$ of homogenous elements of $\mathcal{M}$ such that
	$\bigoplus_{m=1}^N I(v_m) = \mathcal{M}$.

By applying the construction in the proof of Lemma 6 in \cite{corbet2018}, here reported in Definition \ref{def:presentation-matrix} in \cref{sec:from-pers-module-to-matrix}, one can always turn a persistence module into a graded module presentation. Once a persistence module is assigned a presentation matrix, the graded Smith normal form reduction proposed in~\cite{skraba+2013x} (reported in \cref{alg:graded-SNF} in \cref{sec:algo-graded-normal}) returns an interval basis. Details will be treated in~\Cref{sec:graded-modules}.

As a guiding example, consider the persistence module $\mathcal{M}=\{ (M_i,\varphi_i) \}_{i=1}^{3}$ with coefficients in a field $\F$ and structure maps 

\begin{equation} \label{ex:RunningExPersMod}
\mathcal{M}:  0 \ \underset{\scalebox{0.6}{$\begin{pmatrix} 0 \end{pmatrix}$}}{\overset{\varphi_0}{\longrightarrow}} \ \F \ \underset{\scalebox{0.6}{$\begin{pmatrix} 1 \\ 0 \end{pmatrix}$}}{\overset{\varphi_1}{\longrightarrow}} \ \F^2 \ \underset{\scalebox{0.6}{$\begin{pmatrix} 1 & 1 \end{pmatrix}$}}{\overset{\varphi_2}{\longrightarrow}} \ \F \ \underset{\scalebox{0.6}{$\begin{pmatrix}  0 \end{pmatrix}$}}{\overset{\varphi_3}{\longrightarrow}} \ 0, 
\end{equation}

so $M_1 \cong M_3 \cong \F$ and $M_2 \cong \F^2$. The decomposition $ \mathcal{M} \cong \bigoplus_{m=1}^N I_{[b_m, d_m]}$ is then (up to isomorphism) given by: 

\begin{equation}
(0\underset{0}{\to} \F \underset{1}{\to}\F \underset{1}{\to}\F\underset{0}{\to}0)\oplus(0\underset{0}{\to} 0 \underset{0}{\to} \F \underset{0}{\to} 0\underset{0}{\to}0).
\label{eq:ex-quiver-decomposition}	
\end{equation}

Consider now, $v_1=(1)\in M_1$ and $v_2=(0,1)^\top\in M_2$, one has:

\[ \ v_1= \begin{pmatrix} 1 \end{pmatrix} \ {\overset{\varphi_1}{\longrightarrow}} \ \begin{pmatrix} 1 \\ 0 \end{pmatrix} \ {\overset{\varphi_2}{\longrightarrow}} \ \begin{pmatrix} 1 \end{pmatrix} \ {\overset{\varphi_3}{\longrightarrow}} \ \begin{pmatrix} 0 \end{pmatrix}; \]

\[  \ v_2= \ \begin{pmatrix} 0 \\ 1 \end{pmatrix} \ {\overset{\varphi_2}{\longrightarrow}} \ \begin{pmatrix} 1 \end{pmatrix} \ {\overset{\varphi_3}{\longrightarrow}} \ \begin{pmatrix} 0 \end{pmatrix}. \]

A minimal presentation (see Definition \ref{def:presentation} in \cref{subsec:graded-modules}) of the associated graded $\F[x]$-module $\alpha(\mathcal{M})$ is thus obtained as the cokernel of the presentation matrix:

\begin{equation}
S = 
\begin{pmatrix}      x^2 & x^3  \\ 
                    -x & 0 \end{pmatrix},    
    \label{eq:presentation-classical}
\end{equation}
\noindent
whose columns correspond to the (homogeneous) independent relations satisfied by the homogeneous generators $v_1, v_2$ of $\alpha(\mathcal{M})$: that is $x^2v_1=xv_2$ ($\deg=3$) and $x^3v_1=0$ ($\deg=4)$. The elements $v_1$ and $v_2$ form a minimal system of generators (see Definition \ref{def:presentation} in \Cref{subsec:graded-modules}) for $\alpha(\mathcal{M})$, nevertheless they do not form an interval basis for $\alpha(M)$ because $I(v_2)=0\underset{0}{\to} 0 \underset{0}{\to} \F \underset{1}{\to} \F\underset{0}{\to}0$ does not appear in  the decomposition \eqref{eq:ex-quiver-decomposition}.

On the contrary, we obtain an interval basis using $v'_1=v_1=(1)\in M_1$ and $v'_2=(-1,1)^\top\in M_2$.
Considering that $x^kv'_1=0$ iff $k\geq3$ and $xv'_2=0$, the corresponding presentation matrix is the following
 
\begin{equation}
\begin{pmatrix}      0 & x^3  \\ 
                        x & 0 \end{pmatrix}    
    \label{eq:presentation-interval}
\end{equation}
and $v'_1$ and $v'_2$ form an interval basis.
This basic example shows that not all the minimal systems of homogeneous generators of a graded module over $\F[x]$ are interval bases, while, a fortiori, the opposite is true.
Indeed, the presentation associated with an interval basis presents a particular kind of relation; each relation involves a single generator up to multiplication by a homogeneous element in $\F[x]$ as exemplified above. 
In other words, an interval basis is a minimal presentation whose presentation matrix is in Smith normal form (see \Cref{thm:snf-interval-correspondence} in \Cref{subsec:graded-modules}).

The main result of this paper is to present \cref{alg:total-decomposition} to find an interval basis of $\mathcal{M}$ without computing a presentation of $\alpha(\mathcal{M})$. 
Our algorithm is distributed over persistence module steps (\cref{alg:ssd-general}) and avoids explicitly constructing a presentation matrix.
A specialization to the case of real coefficient is included in \cref{alg:ssd-real} in \cref{subsec:real}.

\FloatBarrier

\begin{center}
\vspace{.5cm}
\begin{tikzcd}
{\overset{\text{persistence module}}{\{(M_i \ , \ \varphi_i) \}_{i=1}^{n}}} \arrow[dd, "\text{\cref{alg:total-decomposition}}"'] \arrow[rd, "\text{\cref{def:presentation-matrix}}" ] & \\
& \overset{\text{presentation matrix}}{\Phi} \arrow[ld, "\text{\cref{alg:graded-SNF}}" ] \\
\overset{\text{interval basis}}{\{ v_m \}_{m=1}^N} &                                                   
\end{tikzcd}
\vspace{.5cm}
\end{center}

\FloatBarrier

\subsubsection*{The case of persistent homology modules}

A finite sequence of chain complexes $(C_\bullet^i, \partial_{\bullet}^i)$ for $i\in [n]$, connected by chain maps $f^i: C^i_\bullet \rightarrow C^{i+1}_\bullet$, determines a persistence module $C_k=\{ (C_k^i, f_k^i) \}_{i=1}^n$ for each $k$.
Another persistence module can be obtained by applying to $C_\bullet$ the homology functor in some degree.
Here, we call $\tth k$-{\em persistent homology module} the persistence module $H_k=\{ (H_k^i, \widetilde{f}_k^i) \}_{i=1}^n$ obtained by applying to $C_\bullet$  the homology functor in degree $k$.

Hence in a persistent homology module we do not assume the maps $f^i$ to be necessarily injective, which is a typical assumption within the TDA context.
%
%
\begin{figure}
    \centering
    \subfloat[Step 0]{
    \includegraphics[width = 0.25\textwidth]{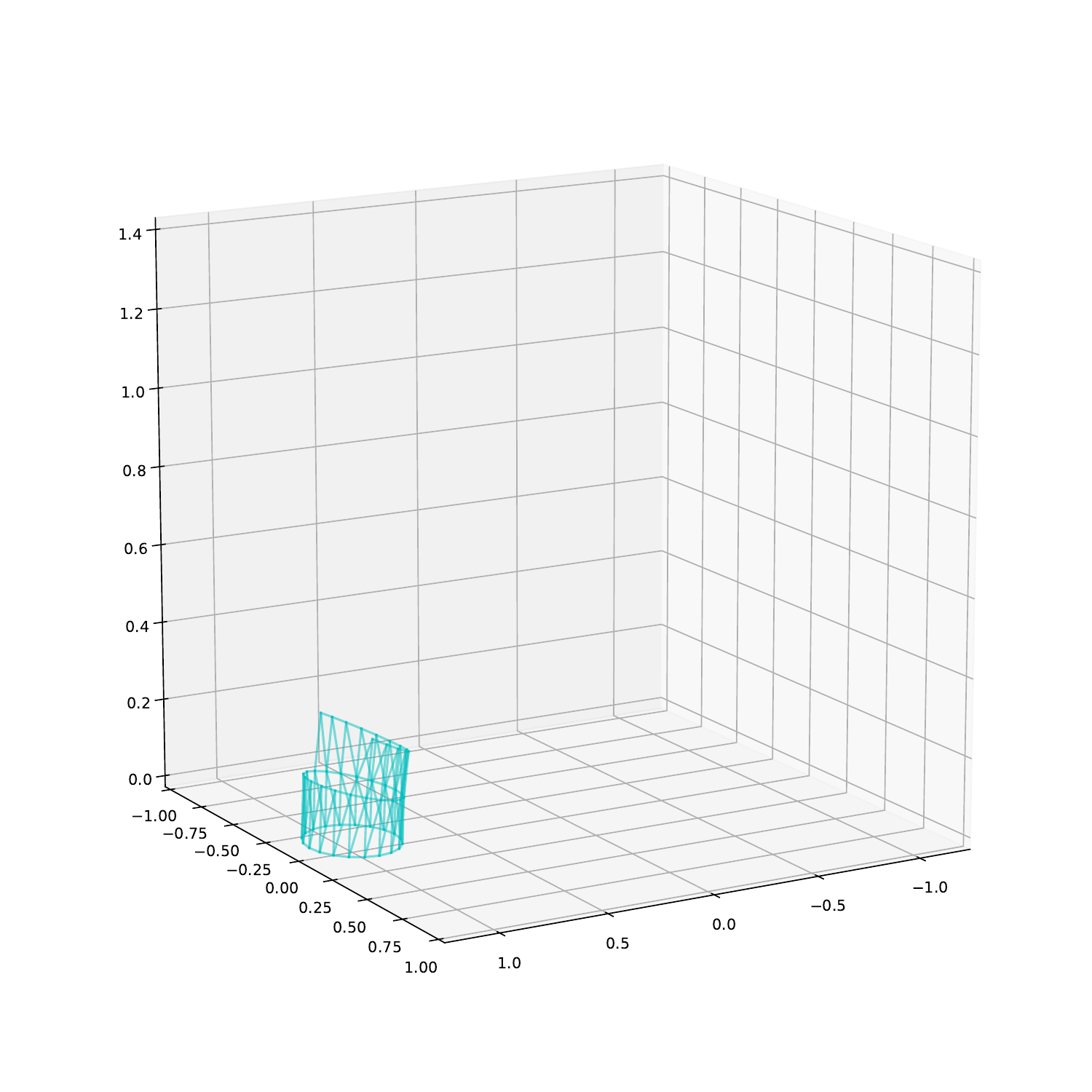}
    }
    \subfloat[Step 1]{
    \includegraphics[width = 0.25\textwidth]{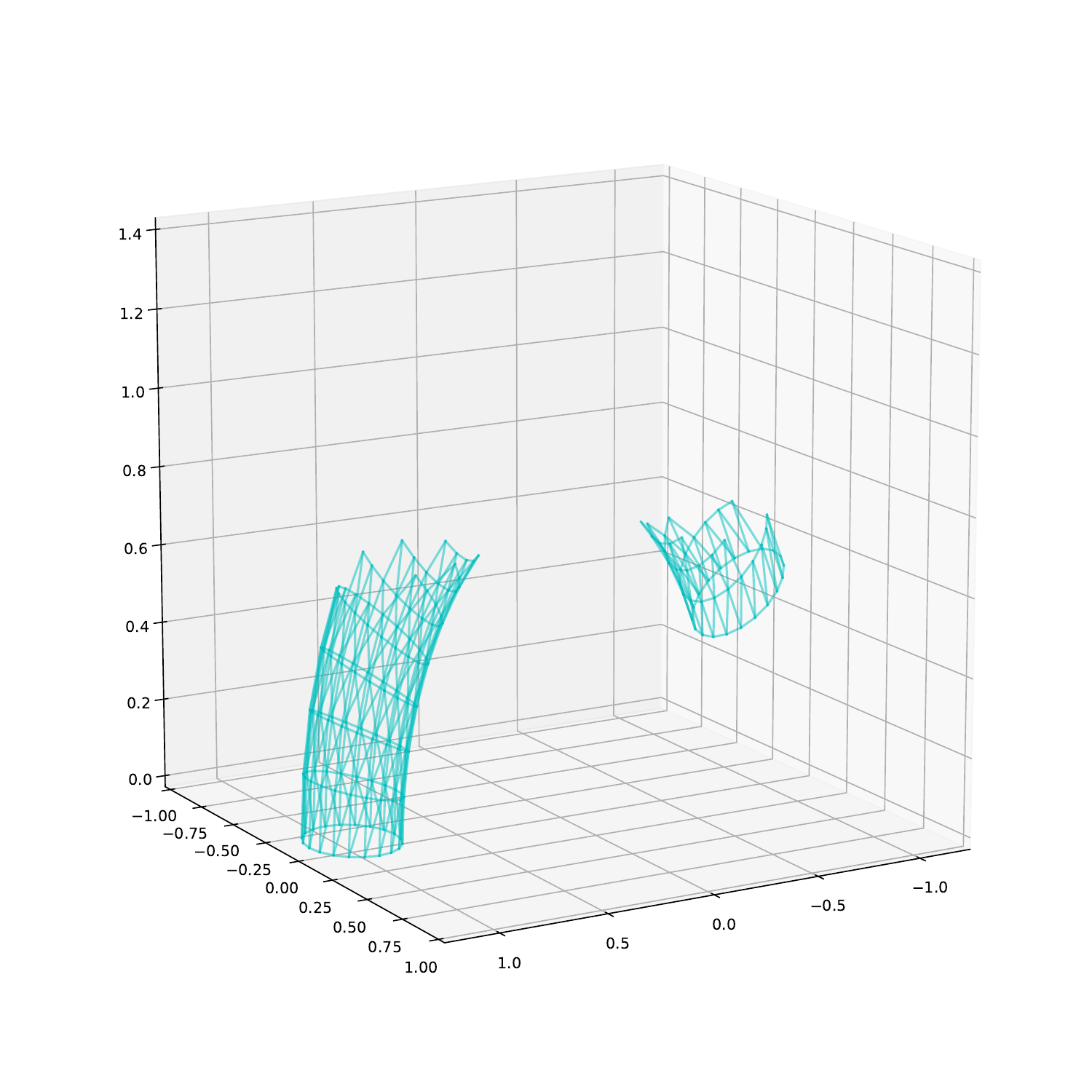}
    }
    \subfloat[Step 2]{
    \includegraphics[width = 0.25\textwidth]{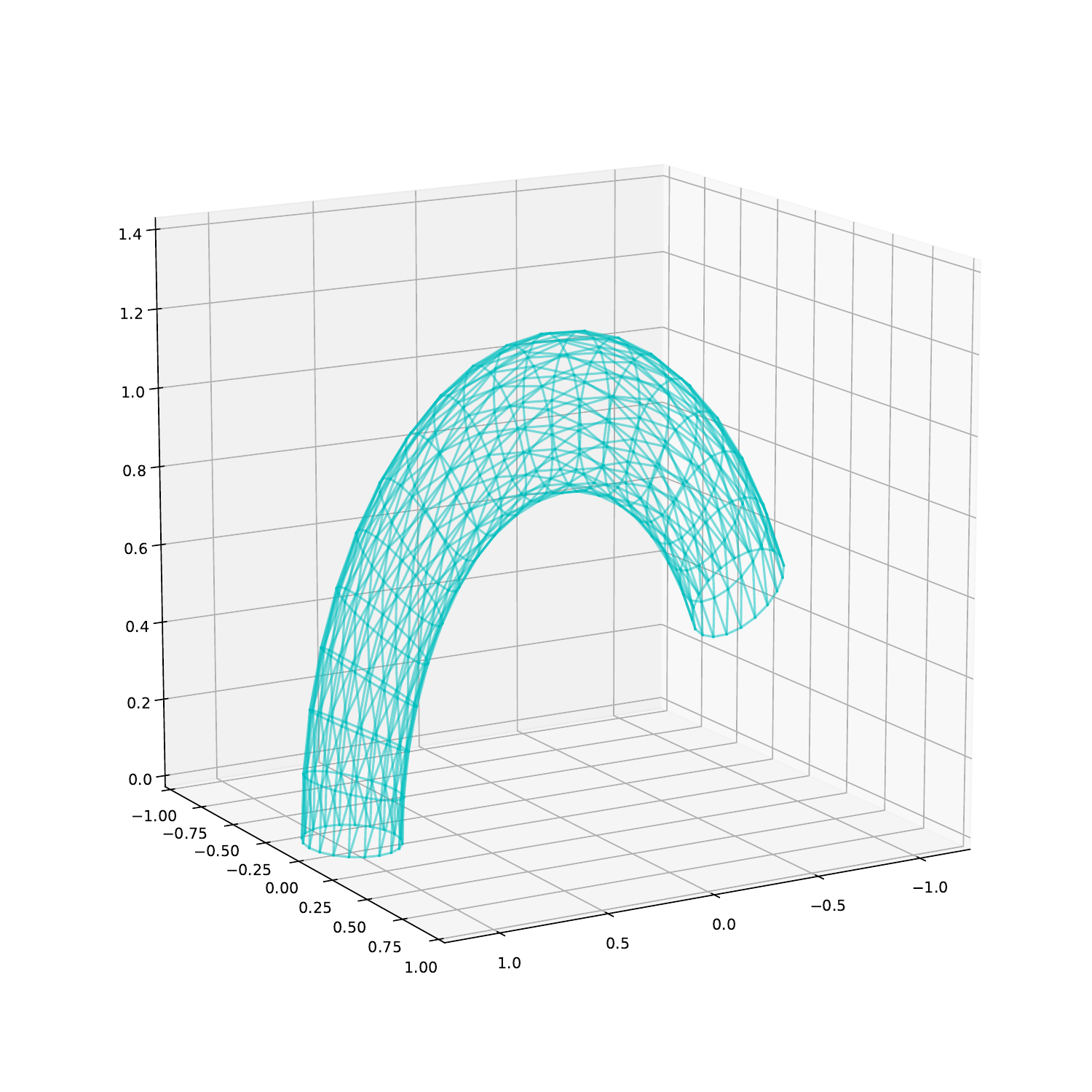}
    }
    \caption{A 3-step filtration by sublevel sets, for the $z$ coordinate, of a tilted and triangulated half torus.}
    \label{fig:example_filtration}
\end{figure}
Indeed, TDA often focuses on the special but relevant case of computing persistence modules from filtered data, such as {\em filtered simplicial or cubical complexes}.
For convenience, when we consider persistent homology in the special case of filtered simplicial complexes (see \Cref{subsec:filtered}), we call it {\em persistent simplicial homology} to avoid confusion.
In that case, the chain maps $f^i$ are assumed to be injective and the combinatorial simplicial structure plays a key role in the computational optimizations and the tracking of homology representatives. The interested reader is referred to \cite{carlsson2009data,ghrist2008barcode} for classical surveys on TDA, and to \cite{Patania2017,vaccarino2022higher,chazal2022introduction} for more recent ones.

Our parallel decomposition algorithm applies to persistent simplicial homology to track homology representatives along a filtered simplicial complex. 

The problem of tracking homology representatives along filtered complexes has been studied mainly from a minimality perspective (see \cite{dey2018efficient,kalivsnik2019higher,Guerra2021HomologicalBases}) to locate the persistent homology features geometrically.
Using the standard algorithm \cite{Carlsson2005DCG} for computing persistent homology, homology representatives can be tracked by storing the operations performed during matrix reduction.
Computing intervals and tracking homology representatives have been optimized in many ways (see \cite{desilva2011dualities,Chen2011,Boissonat2013,fugacci2014efficient,mrozek2009coreduction,Dotko2014}), including parallel and distributed approaches, such as \cite{Bauer2014,Bauer2014distributed,Lipsky2011,Milosavljevic2011}. This list is far from being exhaustive. 
However, not all the mentioned approaches provide an interval basis, and for this purpose, we include the discussion on two relevant cases in Remark \ref{rem:to-reviewer-1} and Remark \ref{rem:to-reviewer-2}.

With respect to the filtered simplicial complex case, our aim is not to outperform computations in persistent simplicial homology but to capture and formalize the algebraic properties of the tracked homology representatives.
In~\cref{fig:example_filtration}, we see an example of a simplicial complex obtained as a triangulation of a portion of a torus filtered by the height function into three steps. 

In the same example of filtered simplicial complex, one can check that the persistence module isomorphism class in the examples of \labelcref{eq:presentation-classical} and \labelcref{eq:presentation-interval} is the class of the persistent homology module obtained by applying the $\st{1}$-homology functor to the filtered simplicial complex in \cref{fig:example_filtration}. Further, generators $v_1$ and $v_2$, which do not form an interval basis, are those associated with the homology classes of the representatives shown in red in \cref{fig:example_classical}, whereas $v'_1$ and $v'_2$ are associated with the homology classes of the representatives in \cref{fig:example_interval}, and do form an interval basis.

\begin{figure}
    \centering
    \subfloat[Representative 1]{
    \includegraphics[width = 0.25\textwidth]{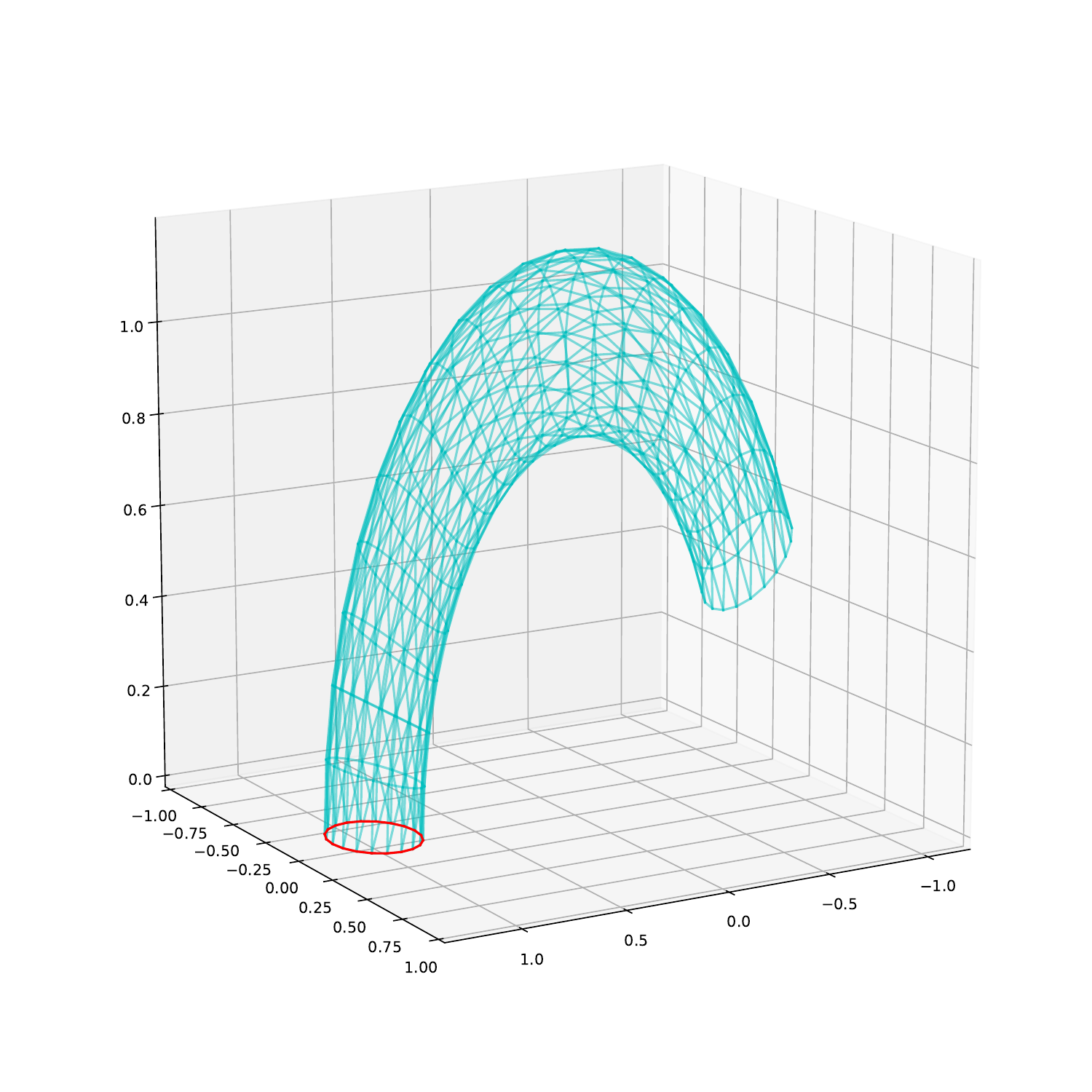}
    }
    \subfloat[Representative 2]{
    \includegraphics[width = 0.25\textwidth]{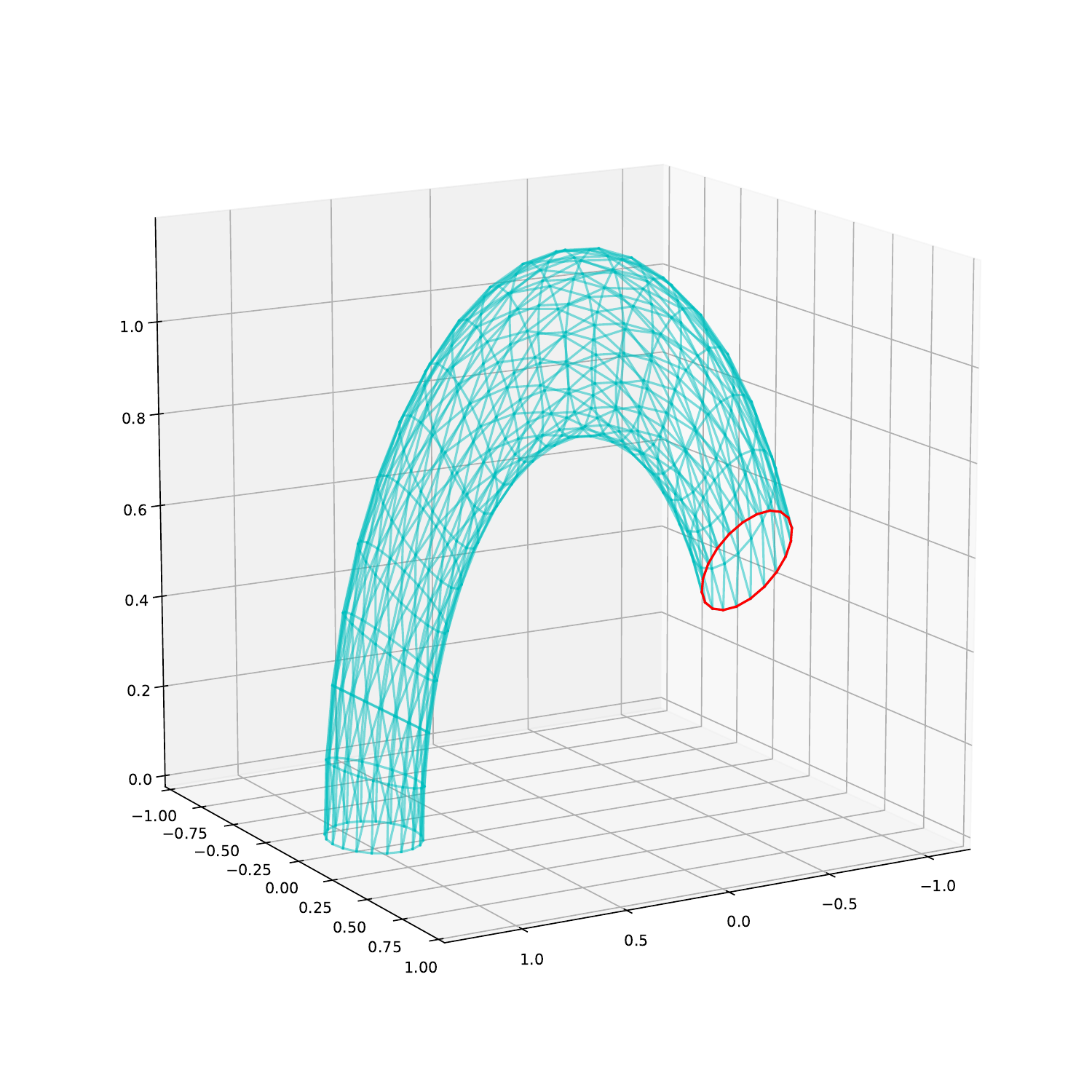}
    }
    
    \caption{In red, two representative $1$-cycles. Their homology classes form a minimal system of generators of the persistent homology module of the filtration in \cref{fig:example_filtration}. These representatives do not induce an interval basis.}
    \label{fig:example_classical}
\end{figure}

\begin{figure}
    \centering
    \subfloat[Representative 1]{
    \includegraphics[width = 0.25\textwidth]{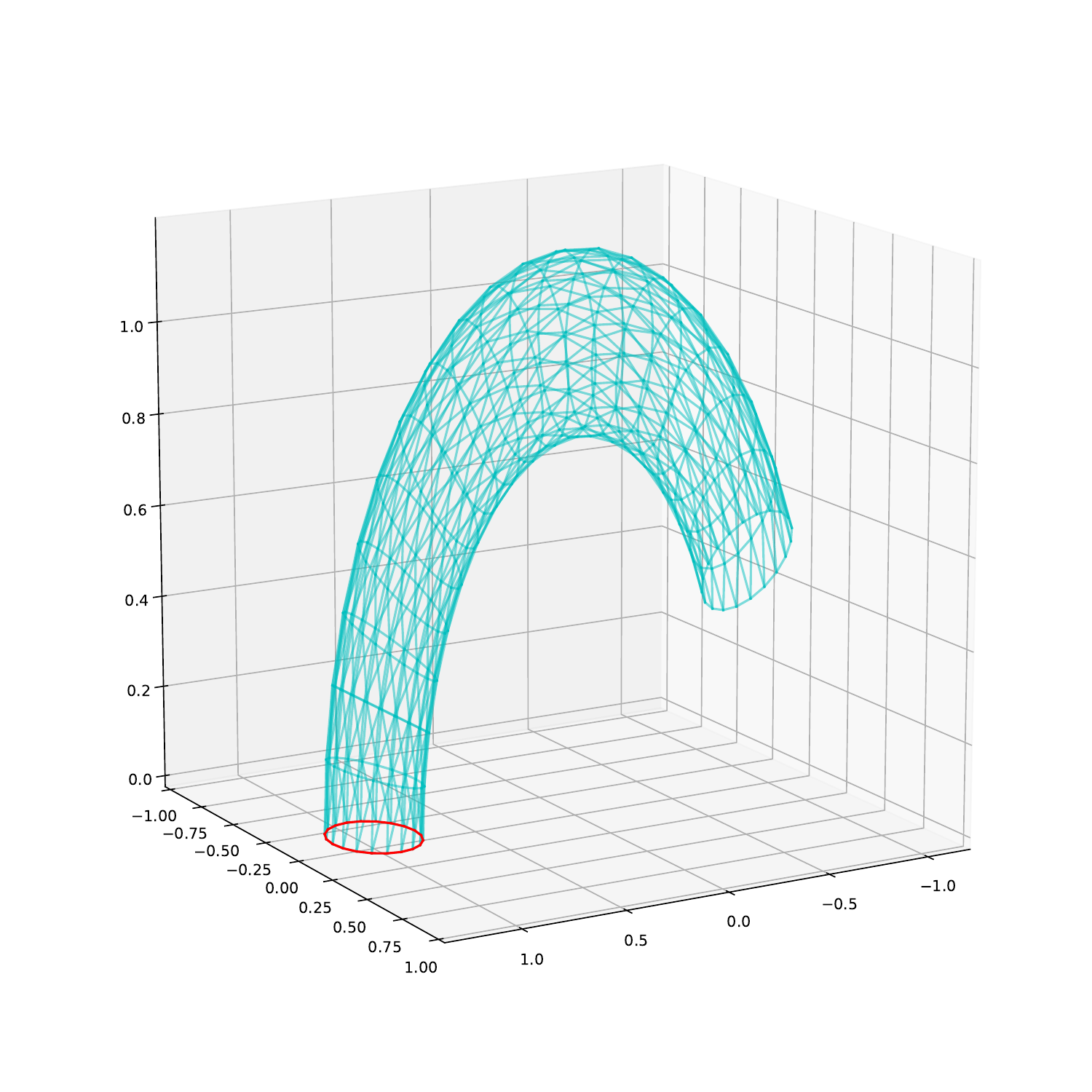}
    }
    \subfloat[Representative 2]{
    \includegraphics[width = 0.25\textwidth]{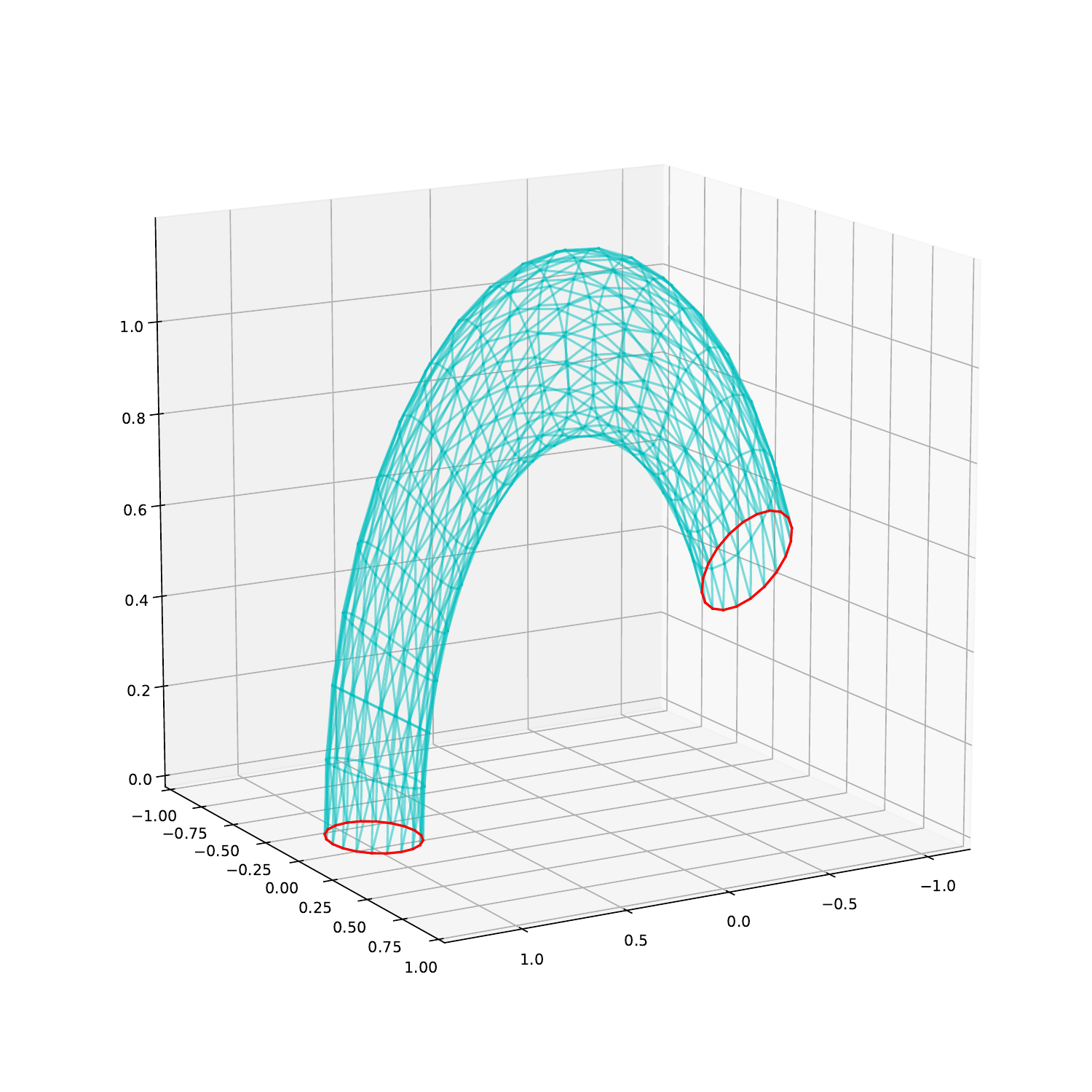}
    }
    
    \caption{In red, a different choice of representative $1$-cycles. Their homology classes are a different choice of generators for the same persistence module as in \cref{fig:example_classical}. However, these generators do induce an interval basis.}
    \label{fig:example_interval}
\end{figure}

As already mentioned, we are not limited to persistent simplicial homology. Indeed, a persistent homology module does not necessarily come from injective simplicial maps. 
In the context of TDA, tracking homology representatives along a monotone (equioriented) sequence of simplicial maps that are not necessarily injective is treated in~\cite{Dey2014ComputingMaps}. 
Each simplicial map is interpreted as a sequence of inclusions and vertex collapses, and a consistent homology basis can be maintained efficiently through a specific data structure called annotations. Later in \cite{Kerber2019BarcodesHomology}, a variation of the coning approach of \cite{Dey2014ComputingMaps} is proposed, which takes a so-called {\em simplicial tower} and converts it into a filtration while preserving its barcode, with asymptotically small overhead. 
Here, we can tackle the same problem from the unifying perspective of persistence modules, thus avoiding specific data structures or reducing them to filtrations.

Finally, our proposed parallel decomposition algorithm applies to tracking {\em harmonic} homology representatives. 
This furthers the recent trend of exploring the interplay between topological data analysis and the properties of the Hodge Laplacian (see \cite{Ebli+2019icmla,Wang+2020PSpecGraph,Wang2021hermes}).
Harmonic homology representatives corresponding to interval bases and computed by our methods for the filtered complex in \cref{fig:example_filtration} are depicted in \cref{fig:example_harmonics}.

\begin{figure}
    \centering
    \subfloat[Representative 1]{
    \includegraphics[width = 0.25\textwidth]{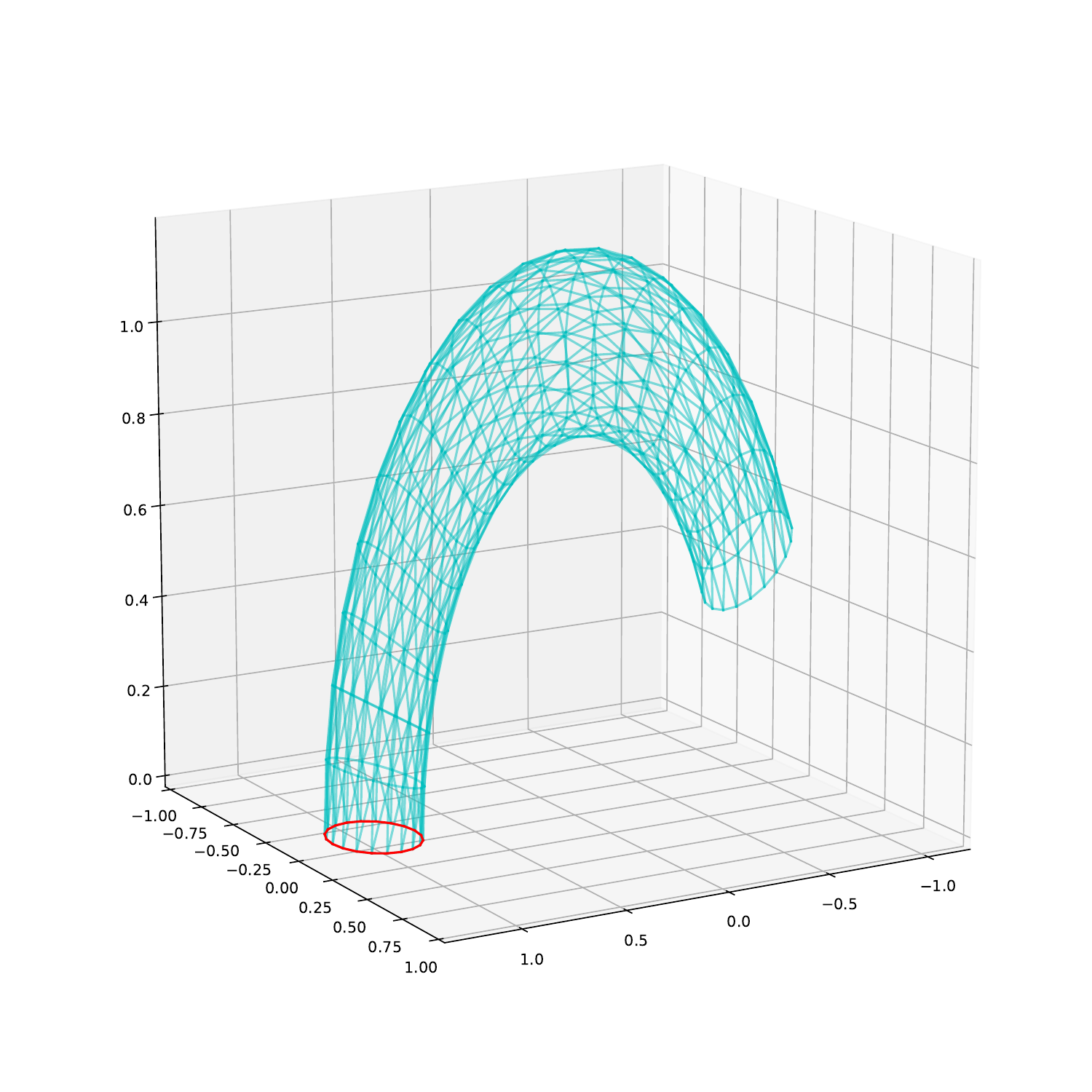}
    }
    \subfloat[Representative 2]{
    \includegraphics[width = 0.25\textwidth]{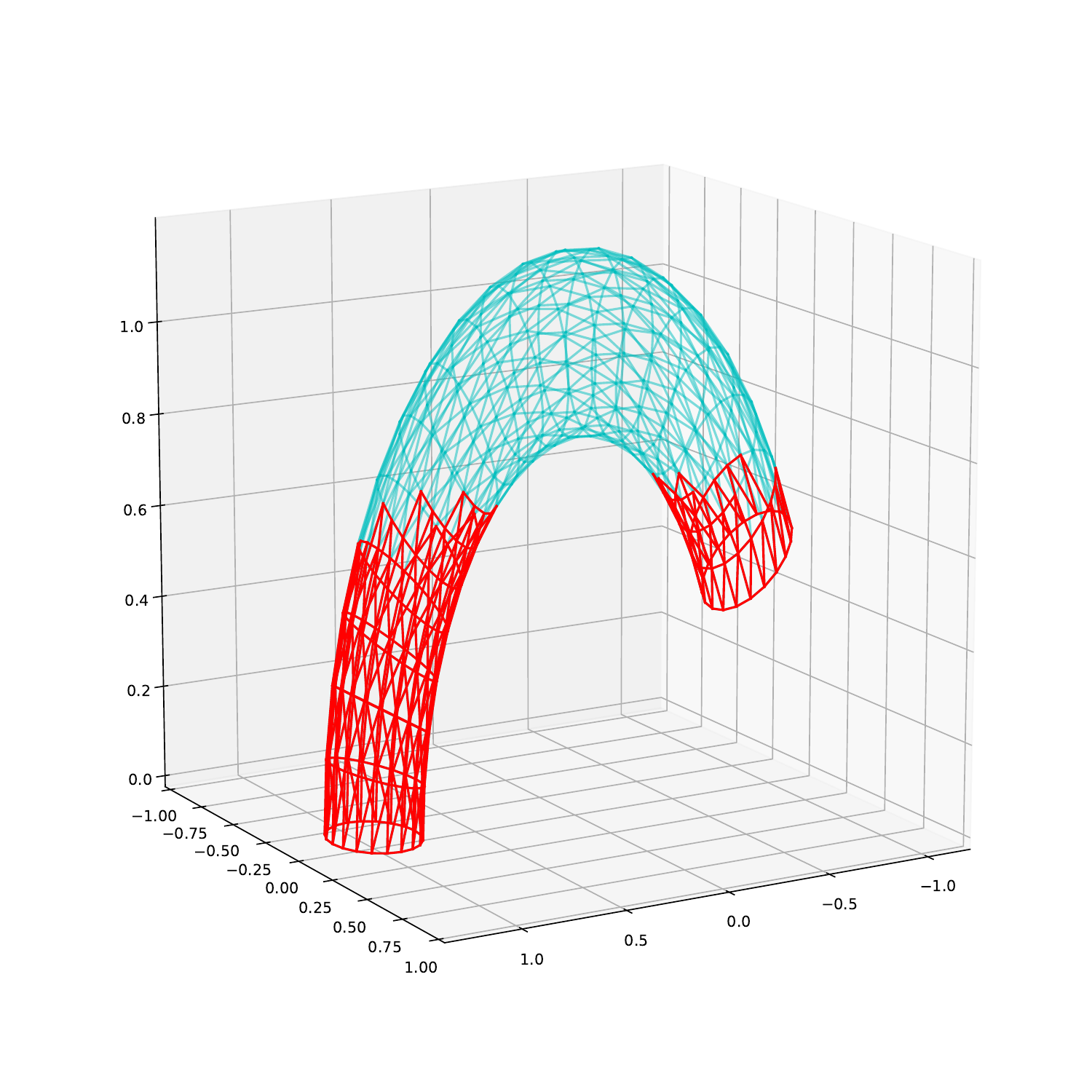}
    }
    
    \caption{Harmonic representatives via the interval basis algorithms}
    \label{fig:example_harmonics}
\end{figure}

\paragraph{Contents.}
In \Cref{sec:Background} and \Cref{sec:graded-modules}, we introduce our notation by formalizing an interval basis as a particular minimal system of generators translated into persistence module terms. We also express the classical interval decomposition result into interval basis terms.
In \cref{sec:related}, we review the literature in the decompositions of graded and persistence modules as quiver representations of type $A_n$.
In \Cref{sec:Decomposition}, we propose an algorithm computing an interval basis out of a persistence module by acting in a distributed way over each step in the input persistence module, and avoiding a presentation of the associated graded module. The same algorithm is specialized to the case of real coefficients in \cref{subsec:real}. This is particularly relevant for the case of harmonics, later discussed in \cref{subsec:harmonics}.
In \Cref{sec:complexity}, we compare the computational cost of our parallel method to the classical Smith Normal Form reduction (whose pseudocode for the graded case is reported in \cref{sec:algo-graded-normal}), when specialized to presentation matrices of persistence modules. In particular, our parallel approach admits an output-dependent estimate, quantifying the advantage of working in parallel.
In \Cref{sec:PersistentHomology}, we describe how to construct in parallel a persistence module from the homology of a monotone sequence of chain maps with homology representatives. Complementary pseudocodes are included in \cref{subsec:homology-algorithms}.
In particular, in \Cref{subsec:harmonics}, we construct in parallel a persistence module from the homology of a monotone sequence of chain maps with harmonic homology representatives. The case of simplicial complex chains is treated in \cref{subsec:simplicial}.


\section{Persistence modules} \label{sec:Background}
In this section, we fix the notation for persistence modules and define interval bases.
In Proposition \ref{prop: interval-existence}, we include the well-known decomposition theorem for equi-oriented quiver representations of type $A_n$ conveniently concerning the interval basis definition. 

For the sake of completeness, in~\cref{sec:graded-modules}, we provide further material connecting the interval basis definition to the Smith Normal Form of a module presentation in the isomorphic category of finitely generated graded $\F[x]$-modules, where $\F[x]$ is the graded ring of polynomials with coefficients in $\F$ and a single indeterminate $x$.
To describe a persistence module, we follow the notation $\{(M_i,\varphi_{i})\}_{i=1}^n$ from Definition \ref{def:quiver-representation}, 
 and we define $\varphi_{i,j} : M_i \rightarrow M_j$ with $i<j$, as the composition $\varphi_{j-1} \circ \dots \circ \varphi_i$. 
Furthermore, under the already described equivalence of categories $\alpha$ we transpose to persistence modules several notions applying to graded modules, such as isomorphisms, homogeneous elements, direct sums, generators, and submodules. Additional details on the equivalence of categories $\alpha$ and graded modules can be found in \Cref{subsec:graded-modules}. 




Let $I(v)$, with $v\in M_b$, be the the persistence module  $\{ (I_i(v), \psi_i(v)) \}_{i=1}^n$ defined by

 \begin{equation*}
     I_i(v) = \begin{cases}
     \langle \varphi_{b,i}(v) \rangle \text{ if }i\geq b,\\
     0\qquad\qquad  \text{otherwise},
     \end{cases}
     \psi_{i}(v) = \begin{cases}
     \varphi_{i}|_{\langle \varphi_{b,i}(v) \rangle} \text{ if }i\geq b,\\
     0\qquad \qquad\text{otherwise},
     \end{cases}
     \quad
 \end{equation*}
 
 \noindent
 where the brackets $\langle \cdot \rangle$ denotes the $\mathbb{F}$-linear space spanned by their argument.

Now, define an {\em (integer) interval} $[b,d]$ with $b \leq d$ to be the finite set of integers $i$ with $b\leq i \leq d$.
The {\em interval module} 
$I_{[b,d]}$ relative to the interval $[b,d]$ is the persistence module $\{(I_i, \psi_{i})\}_{i=1}^n$ such that
 \begin{equation*}
     I_i = \begin{cases}
     \mathbb{F}\qquad \text{if } b\leq i \leq  d,\\
     0\qquad \text{otherwise},
     \end{cases}
     \quad
     \psi_{i} = \begin{cases}
     \operatorname{id}_\mathbb{F}\qquad \text{if } b\leq i <  d,\\
     0\qquad\quad \text{otherwise}.
     \end{cases}
     \quad
 \end{equation*}
 
\begin{remark}
	Fix a degree $b$. For each $v\in M_b$, there exists $d\leq n$ such that
	$$
	I(v)\cong I_{[b,d]}.
	$$
	\label{rem:iso-single-interval-module}
\end{remark} 
Indeed, by construction, each step in $I(v)$ is isomorphic to the vector space $\mathbb{F}$ or to $0$.
The structure maps in $I(v)$ are either isomorphisms or the null map. If an integer $r\leq d-1$ exists, such that $I_{r+1}(v)=0$, we take $d$ to be the minimum of such $r$'s. 
Otherwise, $d=n$.

\begin{definition}{(Interval basis)}
	Given a persistence module $\mathcal{M} = \{(M_i,\varphi_{i})\}_{i=1}^n$, a finite family $\{v_1,\dots, v_N\}\subseteq \bigcup_i M_i$ of homogeneous non-zero elements is an {\em interval basis} for $\mathcal{M}$ if and only if
	
	\[
	\bigoplus_{m=1}^N I(v_m) = \mathcal{M}.
	\]
	
	\label{def:interval-basis}
\end{definition}

\begin{proposition}
Every persistence module $\mathcal{M}$ 
admits an interval basis. 
\label{prop: interval-existence}
\end{proposition}

\begin{proof}
The existence of an interval basis for each $\mathcal{M}$ follows from the interval decomposition corresponding to the Structure Theorem~\cite{Carlsson2005DCG} for finitely generated graded $\mathbb{F}[x]$-modules.
Indeed, the interval decomposition implies that $\mathcal{M}$ decomposes into a direct sum of interval modules of the form.

\begin{equation}
	\mathcal{M}\cong \bigoplus_{m=1}^N I_{[b_m,d_m]},
	\label{eq:interval-decomposition}
\end{equation}

\noindent
where the intervals $[b_m,d_m]$ with $b_m\leq d_m\leq n$ are uniquely determined up to reorderings.
Let $\Psi=\{\Psi_i\}_{i=1}^n: \oplus_{m=1}^N I_{[b_m,d_m]} \longrightarrow \mathcal{M}$ be the persistence module isomorphism  of the interval decomposition in~\eqref{eq:interval-decomposition}. Then, for each summand $I_{[b_m,d_m]}$, the map $\Psi_{b_m}$ detects a vector $v_m\in M_{b_m}$.
By Remark~\ref{rem:iso-single-interval-module}, we have that $I(v_m)\cong I_{[b_m,r_m]}$ for some $b_m\leq r_m \leq n$.
Observe now that, for all indices  $i$ such that $(i,i+1)$ is an arrow in $[n]$, the decomposition isomorphisms satisfy $\varphi_i\circ\Psi_i = \Psi_{i+1}\circ\psi_{i}$, where $\psi_i$ is the structure map of $I_{[b_m,d_m]}$.  
This implies that $r_m=d_m$ for all indices $i\in\mathbb{N}$.
\end{proof}

Decomposing a persistence module via an interval basis consists in retrieving, given a persistence module $\mathcal{M}$, an interval basis ${v_1,\dots,v_N}$, where $N$ equals the number of interval modules in the interval decomposition of Definition \ref{def:interval-basis}.

\section{Related works}
\label{sec:related}

The related works comprise methods for the decomposition of persistence and graded modules.

\subsubsection*{Persistence module decompositions}
Persistence module decomposition methods can be seen as special instances of decomposing quivers of type $A_n$, hence holding for the so-called  {\em zig-zag persistence modules}.
The incremental algorithm introduced in \cite{Carlsson2010ZigzagPersistence} retrieves the interval decomposition by restricting, at each step, to intervals that vanish and constructing on them a flag of images of structure maps. The procedure is a dual counterpart to the kernel flag decomposition we propose in this work. Unlike our approach, the zig-zag decomposition does not aim to recover the generators since generators and intervals are not in one-to-one correspondence for general zig-zag persistence modules.
More recently in\cite{Carlsson2019PersistentViewpoint}, a basis suitable for the zig-zag case, called a {\em canonical form}, has been introduced. It differs from an interval basis in that canonical form consists of a vector space basis for each step in the persistence module. Those bases are selected so that the structure maps connecting the spaces are expressed through matrices in echelon form. 
When comparable, i.e. in the case of equi-oriented quivers, an interval basis is equivalent to the canonical form.
Specifically, an interval basis encodes the data of a canonical form in a compressed way, in the sense that we represent a single generator per interval belonging to the interval decomposition. The structure maps are the original ones, implicitly encoded by the action of $x$. 

Canonical forms can be computed as proposed in \cite{Carlsson2019PersistentViewpoint}, where the decomposition of a zigzag module is tackled from the matrix factorization viewpoint. This approach admits a divide-and-conquer implementation, where the module is subdivided into equally-oriented parts.
After the matrix factorization, the interval lengths can be retrieved by connecting the pivots in the factorization.
Instead, our parallel algorithm leverages graded module presentations to focus on generators rather than basis changes. %
Interval basis elements are found already equipped with their associated interval lengths.
Furthermore, our distributed method is not a divide-and-conquer approach; instead, it performs computations independently across all steps in the persistence module. 

What we call interval basis in this work has multiple germanes in the literature, recently introduced with different purposes.
In \cite{jacquard2022}, a notion similar to the canonical form is called {\em barcode basis}. It is introduced to study the space of transformations from one barcode basis to another as a tool to express in barcode basis terms the decomposition of commutative ladders from \cite{escolar2016ladders} and the possibility of defining partial barcode matchings out of a quiver morphism.
In \cite{gonzalez-diaz2021x}, authors introduce {\em persistence bases} as an isomorphism realizing the interval decomposition (see \eqref{eq:interval-decomposition}) to define barcode matchings induced by persistence module morphisms.
Their direct limit construction is based on the descending chain condition discussed in~\cite{crawley-boevey2015}. 
Our flag of kernels over each filtration step corresponds to one of the flag of spaces, namely the positive flag of kernels with respect to a cut, introduced in~\cite{crawley-boevey2015}.

Finally, our subdivision into homogeneous spaces generated by an interval basis specializes the quotient through the radical functor introduced in~\cite{chacholski2016noise}, where authors characterize tameness conditions in multiparameter persistence.


\subsubsection*{Graded module decompositions}
Given a presentation matrix, many methods exist in the literature to retrieve a minimal system of generators for a graded $\F[x]$-module.

As noticed in \cite{lesnick-wright2022bigraded-present}, extracting a minimal system of generators can be seen as a specialization to $\F[x]$ coefficients of classical Gr{\"o}bner basis extraction algorithms \cite{schreyer1980berechnung,faugere1999new,faugere2002new} for multigraded $\F[x_1,\dots, x_n]$-modules, widely implemented in software packages \cite{eisenbud2001computations,greuel2002singular,bosma1997magma,CoCoA}.
See \cite{rutman1992} for Gr{\"o}bner bases of modules and primary decompositions.
As already pointed out in \Cref{sec:intro}, a minimal system of generators is not, in general, an interval basis (see \labelcref{eq:presentation-classical}).

Instead, an interval basis is computable by reducing the presentation matrix into the Smith Normal Form (see \cref{sec:graded-modules}).
To the best of our knowledge, the authors of \cite{skraba+2013x} first introduced, for the graded case, an algorithm for the SNF reduction.
The general procedure's complexity depends solely on the overall size of the presentation matrix. This motivates us to specialize in \cref{sec:gradedSmith} the same procedure to presentation matrices obtained through the construction in \cite{corbet2018}, relevant when starting from a persistence module. 
We provide a complexity estimate for that specific case, which takes advantage of the sparsity given by the block structure in the presentation matrix of a persistence module.
On the contrary, in our parallel decomposition algorithm in \cref{sec:Decomposition}, we take advantage of the independence properties under the action of $x$ of the interval basis to propose a parallel approach with output-dependent complexity.

\section{Parallel computation of an interval basis} \label{sec:Decomposition}

In this section, we present a parallel algorithm for the computation of an interval basis of a persistence module $\mathcal{M}$ (see \cref{alg:ssd-real} \cref{subsec:real} for a specialization to the real coefficient case). 
The whole, length-$n$ persistence module (spaces and structure maps) is assumed to be input to a pool of $n$ processors. Then, each step $M_i$ can be processed independently by processor $i$. The idea is that a \textit{single step decomposition} routine (\cref{alg:ssd-general}) takes care of the bars being born at step $i$; to discriminate these from bars that are merely traveling through step $i$, we make use of the flag of vector spaces given by the kernels of iterated composites of the structure maps. Simple linear algebra then shows that we can recover basis vectors that form an interval basis of $\mathcal{M}$. 
Using such a flag of kernels takes implicit advantage of the death of bars along the barcode, gradually reducing the size of the maps involved and achieving better efficiency than a method that does not take this into account, such as the graded SNF. This statement is justified in the complexity analysis in \cref{sec:complexity}.

Consider a persistence module $\mathcal{M} = \{(M_i,\varphi_i)\}_{i=1}^n$. Without loss of generality, we assume an additional structure map $\varphi_n:M_n\to M_{n+1}$ to be the null one. 
This way, the treatment of the final step $M_n$ has no qualitative difference from the others.
Denote by $m_i$ the dimension of the space $M_i$ and with $r_i$ the dimension of $\im (\varphi_{i-1})$. For each $i$ there is a flag of vector subspaces of $M_i$ given by the kernel of the maps $\varphi_{i,j}$:

\begin{equation}
    0\subseteq \ker (\varphi_{i,i+1}) \subseteq \ker (\varphi_{i,i+2})\subseteq \dots \subseteq \ker (\varphi_{i, n+1}) = M_i,
\end{equation}

\noindent
where the last equality holds by the assumption on the last map above.
Denote for simplicity each space $\ker (\varphi_{i,j})$ as $V_j^i$. 

An \defemph{adapted basis} for the flag in $M_i$ is given by a set of linearly independent vectors $\mathcal{V}^i = \{v_1, \dots, v_{m_i}\}$, and an index function $J:\mathcal{V}^i\to \{1,\dots, n-i+1\}$, such that

\begin{equation}
    V_{i+s}^i = \langle \{v\in \mathcal{V}^i\ |\ J(v) \leq s\} \rangle \quad \forall s,\ 1\leq s\leq n-i+1.
\end{equation}

In words, an adapted basis is an ordered list of vectors in $M_i$ such that for every $j$, the first $\dim V_j^i$ vectors are a basis of $V_j^i$ (an empty list is a basis of the trivial space). The index function $J$ gives precisely this ordering. Notice that
\begin{lemma} \label{doubly_adapted_basis}
Without loss of generality, it is possible to choose an adapted basis for $M_i$ so that it contains a basis of $\im (\varphi_{i-1})$ as a subset.
\end{lemma}

\begin{proof}
Let us consider an adapted basis $\mathcal{V}^i = \left( t_1,\dots, t_{m_i} \right)$ for the flag of kernels in $M_i$, with the vectors $t_1,\dots, t_{m_i}$ ordered by index function $J$. We construct the desired basis explicitly: set $\underline{\mathcal{V}}^i = \{t_1\}$. For every $s=2,\dots, m_i$, if $t_s\notin \langle \underline{\mathcal{V}}^i \rangle + \im(\varphi_{i-1})$ add the vector $t_s$ to $\underline{\mathcal{V}}^i$. Otherwise, it must hold that $t_s = \sum_{a<s}\lambda_at_a + x$ with $x\in \im(\varphi_{i-1})$. Then we add to $\underline{\mathcal{V}}^i$ the vector $x = t_s - \sum_{a<s}\lambda_at_a$. In this way, $\underline{\mathcal{V}}^i$ is another adapted basis, and the elements added by the second route form a basis of $\im (\varphi_{i-1})$.
\end{proof}

Therefore, we shall assume that each basis $\mathcal{V}^i$ is in the form of Lemma \ref{doubly_adapted_basis}. 
Let us introduce two subspaces of $\langle \mathcal{V}^i \rangle$: it holds that $\langle \mathcal{V}^i \rangle  = \langle \mathcal{V}^i_{\text{Birth}} \rangle \oplus \langle\mathcal{V}^i_{\im}\rangle$, where $\mathcal{V}^i_{\im}$ is the subset of $\mathcal{V}^i$ made of a basis of $\im \varphi_{i-1}$, and $\mathcal{V}^i_{\text{Birth}}$ is its complement. 
We aim to construct a basis for the whole persistence module using the adapted bases at each step $i$.
\begin{definition} \label{def:adapted_basis_set}
Let us define $\mathscr{V} := \bigcup_i \mathcal{V}^i_{\text{Birth}}$.
\end{definition}
\noindent
$\mathscr{V}$ is the set of elements of the adapted basis in each degree $i$ that are not elements of $\im (\varphi_{i-1})$. We refer to \cref{subsec:graded-modules} for the definition of degree of an element. In the following, we prove that $\mathscr{V}$ is in fact an interval basis for $\{M_i,\varphi_i\}_{i=1}^n$.

\begin{lemma}
\label{lem:injection-of-restriction}
For any $i<j$, define the set $T := \langle\{v\in \mathcal{V}^i\ |\ J(v)>j-i\}\rangle$.
The restriction $\varphi_{i,j}|_{T}$ of the map $\varphi_{i,j}$ is an injection.
\end{lemma}
\begin{proof}
By definition of $T$, it holds $M_i = \ker (\varphi_{i,j}) \oplus T$. If the restriction of $\varphi_{i,j}$ onto $T$ were not injective, then $T$ and $\ker (\varphi_{i,j})$ would have nontrivial intersection. This is a contradiction.
\end{proof}

\begin{lemma}
\label{lem:intersection-image-new-things}
For any $i<j\in \mathbb{N}$, it holds

\begin{equation}
    \varphi_{i,j}\left(\langle \mathcal{V}^i_{\text{Birth}} \rangle\right) \ \cap \ \varphi_{i,j}\left( \langle \mathcal{V}^i_{\im} \rangle \right) = \{0\}.
\end{equation}

\end{lemma}
\begin{proof}
Suppose that the intersection contains a nonzero vector $u$:

\[
0 \neq u = \varphi_{i,j}\left(\sum_{v_k\in \mathcal{V}^i_{\text{Birth}}}\lambda_k v_k\right) = \varphi_{i,j}\left(\sum_{w_l\in \mathcal{V}^i_{\im}}\mu_l w_l\right).
\]

Denote by $u_\text{B}$ and $u_\text{I}$ the vectors

\[
u_\text{B} = \sum_{\substack{J(v_k)>j-i\\ v_k \in \mathcal{V}^i_\text{Birth}}}\lambda_k v_k, \quad u_\text{I} =  \sum_{\substack{J(w_l)>j-i\\ w_l\in \mathcal{V}^i_{\im}}}\mu_l w_l.
\]

Since all the elements $v$ such that $J(v)\leq j-i$ belong to $\ker(\varphi_{i,j})$, it holds $u = \varphi_{i,j}\left(u_\text{B}\right) = \varphi_{i,j}\left(u_\text{I}\right)$. Then, $u$ is the image through $\varphi_{i,j}$ of an element of $T = \langle\{v \in \mathcal{V}^i\ |\ J(v)>j-i\}\rangle$. On the other hand, also, the difference $u_\text{B} - u_\text{I}$ belongs to the same space and is mapped to zero by $\varphi_{i,j}$. The restriction of $\varphi_{i,j}$ to $T$ is injective because of Lemma \ref{lem:injection-of-restriction}, therefore it must be $u_\text{B} - u_\text{I} = 0$. 
Since $\langle \mathcal{V}^i \rangle  = \langle \mathcal{V}^i_{\text{Birth}} \rangle \oplus \langle\mathcal{V}^i_{\im}\rangle$, it must be $u_\text{B} = u_\text{I} = 0$, hence $u=0$.
\end{proof}
We move now to the main theorem of this section. For convenience (and without loss of generality) we can assume that $\mathcal{V}^0$ is empty, and that both $M_k$ and $\varphi_k$ are trivial for $k\leq 0$. 
\begin{theorem}
    The set $\mathscr{V}$ is an interval basis for the persistence module $M$.
    \label{th:correctness}
\end{theorem}
\begin{proof}
Say that $\mathscr{V} = \{v_1, \dots, v_N\}$. Each vector $v_j$ in the set $\mathscr{V}$ induces an interval module $I(v_j)$. We want to show that $M = \bigoplus_{j=1}^NI(v_j)$. To do so, let us see that for each $0\leq i\leq n$, the space $M_i$ is exactly  $\bigoplus_{j=1}^NI_i(v_j)$. By construction, we know that

\begin{equation}
    M_i = \im (\varphi_{i-1}) \oplus \langle\{v \in \mathcal{V}^i\ |\ v\notin \im (\varphi_{i-1})\}\rangle \  = 
    \im (\varphi_{i-1})
    \oplus
    \bigoplus_{\substack{v \in \mathscr{V}\\ \deg v =i}}I_i(v)
\end{equation}

\noindent
All we have to show is that $\im (\varphi_{i-1})$ can be written as $\bigoplus_{\substack{v \in \mathscr{V}\\ \deg v <i}}I_i(v)$. First, we will see that a sum decomposition holds.
By definition, 
an element in the sum belongs to the image of $\varphi_{i-1}$.
We can show the converse by induction over the step-index $i\in \N$.
For $i=0$, consider $M_0 = \langle \mathcal{V}^0 \rangle$.
None of the elements of $\mathcal{V}^0$ belongs to $\im(\varphi_{-1})$.
It holds that the image of $\varphi_0$ is contained in the sum as desired.
Suppose by induction that for any $i-1$, the image of $\varphi_{i-2}$ is contained in the sum.
Then, since $M_{i-1} = \langle \{v \in \mathcal{V}^{i-1}\ |\ v\notin \im (\varphi_{i-2})\}\rangle \ \oplus \ \im (\varphi_{i-2})$, it holds that

\begin{equation}
    \im (\varphi_{i-1}) = \sum_{\substack{v \in \mathcal{V}^{i-1}\\ v\notin \im (\varphi_{i-2})}} I_i(v) + \varphi_{i-1}(\im (\varphi_{i-2})).
    \label{eq:image-sum-here}
\end{equation}

Therefore, by the induction hypothesis, we have that.

\begin{equation*}
    \im(\varphi_{i-1}) \subseteq \sum_{\substack{v \in \mathscr{V}\\ \deg v< i}} I_i(v).
\end{equation*}

Now that we have shown the sum decomposition, it remains to be seen that this sum is direct.
We proceed iteratively from degree $i-1$ down to $0$.
We show that only a trivial combination of elements in the sum gives the null element in $\im(\varphi_{i-1})$. 
By \labelcref{eq:image-sum-here}, we know that there exists $x\in \im(\varphi_{i-2}) \subseteq M_{i-1}$ along with a finite number of coefficients $\lambda_r$ and interval basis elements
$v_r\in \mathcal{V}^{i-1}$, such that $v_r\notin \im (\varphi_{i-2})$, so that
\begin{equation*}
    0 = \varphi_{i-1}(x) + \sum_{r} \lambda_r \varphi_{i-1}(v_r) .
    \label{eq:image-sum}
\end{equation*}
Because of Lemma \ref{lem:intersection-image-new-things}, it must be
\[
\varphi_{i-1}(x) = \sum_{r} \lambda_r \varphi_{i-1}(v_r) = 0.
\]
Let us now focus on the second equality.
We can restrict to indices $r$ such that  $\varphi_{i-1}(v_r) \neq 0$. 
If there are no such indices, there is nothing to prove.
%
%
As all summands are different from zero, the index $J(v_{r})$ of the vectors in the adapted basis has to be greater than $1=i-(i-1)$. Hence, because of Lemma \ref{lem:injection-of-restriction}, it holds that $\sum_{r}\lambda_rv_r = 0$. Since the elements in $\{v_{r} | \deg v_{r} = i-1\}$ are linearly independent it must be $\lambda_r= 0$ for any $r$.

We iteratively repeat the same reasoning for $\varphi_{i-2}(x)=0$ where $x$ has now degree $i-2$. Since there are finitely many vectors, this process has an end, and this concludes our proof.





\end{proof}

We now explicitly construct the set $\mathscr{V}$. We must first obtain sets $\mathcal{V}^i$ to do so. 
\begin{remark}
Notice that the construction of each $\mathcal{V}^i$ is independent from the others. Therefore, they can be computed simultaneously. 
\end{remark}

\subsubsection*{Construction of $\mathcal{V}^i_\text{Birth}$} 
\label{par:single-space-decomposition}
We first recall that a simple basis extension algorithm is given by the procedure described in the following \cref{alg:basis-extension}.  
\begin{algorithm}
\SetAlgoLined
\textbf{Input}: linearly independent vectors $\mathcal{U}=\{u_1,\dots, u_r\}$, linearly independent vectors $\mathcal{W} = \{w_1,\dots, w_n\}$ \;
\KwResult{minimal set of vectors $w_{i_1},\dots, w_{i_p} \notin \langle\mathcal{U}\rangle$ such that $\langle \mathcal{U}\cup \{w_{i_1},\dots, w_{i_p}\} \rangle=\langle\mathcal{U}\cup \mathcal{W}\rangle$ }
$\mathcal{R} = \{\}$\;
\For{i=1,\dots, n}{
    \If{$\rank (\mathcal{U}) < \rank (\mathcal{U}\cup \{w_i\})$}{
        $\mathcal{U} = \mathcal{U} \cup \{w_i\} $\;
        $\mathcal{R} = \mathcal{R} \cup \{w_i\} $\;
    }
}
\Return{$\mathcal{R}$}
\caption{Basis completion algorithm}
\label{alg:basis-extension}
\end{algorithm}
The set $\mathcal{W}$ is ordered, and its elements are added to $\mathcal{U}$ in their ascending order in $\mathcal{W}$, so that $\mathcal{U}$ is extended to a basis of $\langle\mathcal{U}\rangle + \langle\mathcal{W}\rangle$. In the following, we refer to the extension of basis $\mathcal{U}$ by the vectors in set $\mathcal{W}$ through \cref{alg:basis-extension} as \texttt{bca}$(\mathcal{U},\mathcal{W})$. 

Secondly, we describe \cref{alg:col-red}, performing the standard left-to-right column reduction on matrix $R$. We include the pseudo-code to illustrate the input and output representations needed, namely matrix $C$ and the indices of the zeroed-out columns, so as to reduce the size of the matrices treated in \cref{alg:ssd-general} by ignoring the known zero columns. 

\begin{algorithm}[!]
\SetAlgoLined
\textbf{Input}: $a\times b$ matrix $R$,
$b\times b$ matrix $C$, $I \subseteq \{ 1,\dots, b \}$ set of indices of the zero columns \\
\textbf{Result}: column-reduced $R$, change of basis matrix $C$ , indices of the new zero columns $I'$

$I' \leftarrow \{ \}$ \;
\For{$i \in \{ 1, \dots, b \}$}{
$r_i \leftarrow$ the $\tth{i}$-column in $R$ \;
\If{$i \notin I$}{
$L \leftarrow \{ 1, \dots, i-1 \} \setminus (I \cup I') $ \;
Perform left-to-right reduction of $r_i$ using columns $r_j$, for $j\in L$ \;
Perform the same column operations on $C$ \;
\If{$r_i=0$}{ $I' \leftarrow I' \cup i$}
}
}
\Return{R, C, I'}
 \caption{Column reduction}
 \label{alg:col-red}
\end{algorithm}

\begin{algorithm}[!]
\SetAlgoLined
\textbf{Input}:  map $\varphi_{i-1}:M_{i-1}\to M_i$ ;  \hfill \tcp{matrix w.r.t. fixed bases} 
maps $\{\varphi_j:M_j\to M_{j+1})\}, i\leq j\leq n$ ; \hfill \tcp{matrices w.r.t. fixed bases}
\KwResult{$\mathcal{V}^i_\text{Birth}$ and its index function $J$}
\vspace{0.4cm}
Reduce $\varphi_{i-1}$ and find a basis $\mathcal{U} = \{u_1,\dots, u_k\}$ of $\im (\varphi_{i-1})$\;
$k := \dim \im(\varphi_{i-1})$\;
$R \leftarrow \operatorname{Id}:M_i\to M_i$; \hfill \tcp{Initialize identity matrix}
$C \leftarrow \operatorname{Id}:M_i\to M_i$; \hfill \tcp{Initialize identity matrix}
$r \leftarrow \dim(M_i)$\;
$\mathcal{V}^i_\text{Birth} \leftarrow \{\}$; \hfill \tcp{Initialize empty interval basis}
$\text{inds} \leftarrow \{ \}$; \hfill \tcp{Set of indices of zero columns}
\vspace{0.2cm}

\For(\hfill \tcp*[h]{From $i$ to end}){$s = 0,\dots,n-i$}{
    $R \leftarrow \varphi_{i+s} \cdot R$; \hfill \tcp{Matrix of the map from $i$ to $i+s$}
    $R,C,\text{newInds} \leftarrow$ ColumnReduction($R,C,\text{inds}$)\;
    $r' \leftarrow \rank (R) = r - \vert \text{newInds} \vert $\;
    \If(\hfill \tcp*[h]{If some bar has died}){$r'<r$}{
        $B \leftarrow $ basis of $\ker (R) = \{ Ce_i , i \in \text{newInds} \}$\;
        $B \supseteq B_{\text{new}} \leftarrow \textrm{bca}(\mathcal{U}, B)$; \hfill \tcp{Complete $\mathcal{U}$ to $B$ by $B_{\text{new}}$}
        $\mathcal{U} \leftarrow \mathcal{U}\cup B_{\text{new}}$; \hfill \tcp{Update $\mathcal{U}$}
        $\mathcal{V}^i_\text{Birth} \leftarrow  \mathcal{V}^i_\text{Birth}\cup B_{\text{new}}$; \hfill \tcp{Update interval basis}
        \For{$v \in B_{\text{new}}$}{
        $J(v) \leftarrow s+1$; \hfill \tcp{Set appropriate index} 
        }
        $r \leftarrow r'$; \hfill \tcp{Update "remaining" rank}
        $\text{inds} \leftarrow \text{inds} \cup \text{newInds}$; \hfill \tcp{Update zero columns}
    }
    \tcp{If all bars dead or enough generators} \
    \If{$r=0$ \textbf{or} $|\mathcal{V}^i_\text{Birth}|+k= \dim M_i$}{
    \textbf{break}\;
    }
}
\Return{$\mathcal{V}^i_\text{Birth}$, $J$} 
 \caption{\texttt{ssd} - Single step decomposition of step $M_i$}
 \label{alg:ssd-general}
\end{algorithm}

We next give a general algorithm to construct the set $\mathcal{V}^i_\text{Birth}$ for a given $M_i$ of persistence module $\mathcal{M}$.
To find $\mathcal{V}^i_\text{Birth}$ we need only to iteratively complete a basis of $\ker (\varphi_{i,j})$ to a basis of $\ker( \varphi_{i,j+1})$ within the complement of $\im (\varphi_{i-1})$.
This is done by first computing a basis of $\im (\varphi_{i-1})$, which initializes the basis to be completed.
Then, for each $s$, a basis of $\ker (\varphi_{i,i+s})$ is computed and \Cref{alg:basis-extension} is applied to extend the current basis with the basis of $\ker (\varphi_{i,i+s})$.
%
The entire procedure to construct $\mathcal{V}^i_\text{Birth}$ from $M_i$ and the structure maps is described in \cref{alg:ssd-general}.
In the following, we refer to the construction of $\mathcal{V}^i_\text{Birth}$ through \cref{alg:ssd-general} as \texttt{ssd}$(M_i)$. 
\newpage

\subsubsection*{Construction of $\mathscr{V}$} 
\label{par:persistence-module-decomposition}
Once the decomposition of each space is performed, it is immediate to assemble the interval basis $\mathscr{V}$. Further, 
by storing together with the interval basis the indices of appearance and death of its elements, we can obtain the persistence diagram of module $\{(M_i,\varphi_i)\}_i$ without increasing the computational cost. This is the content of \cref{alg:total-decomposition}, which summarizes the procedures introduced so far into a single routine that takes a persistence module and returns its interval basis and persistence diagram.  

\begin{algorithm}
\SetAlgoLined
\textbf{Input}: persistence module $\{M_i, \varphi_i\}_{i=1}^n$;
    \hfill \tcp{$n-1$ matrices w.r.t. fixed bases}
\KwResult{interval basis $\mathscr{V}$ and persistence diagram}

$\varphi_0:= $ empty matrix with $\dim M_0$ rows and $0$ columns\;
$\varphi_{n+1}:= $ empty matrix with $0$ rows and $\dim M_n$ columns\;

$\mathscr{V} = \{\}$\;
$PD = \{\}$\;
\Parfor($i = 1,\dots, n+1$ ){
    $\mathcal{V}^i_\text{Birth}, J = \operatorname{ssd}(\varphi_{i-1}, \{\varphi_j\}_{j\geq i})$\;
    \For{$v\in \mathcal{V}^i_\text{Birth}$}{
    $\mathscr{V} \leftarrow \mathscr{V} \cup v$ \;
    $PD \leftarrow PD \cup (i, i+J(v))$\;
    }
}
\Return{$\mathscr{V}$, $PD$}
 \caption{Persistence module decomposition}
 \label{alg:total-decomposition}
\end{algorithm}

\FloatBarrier
\noindent
We refer to the decomposition of \cref{alg:total-decomposition} as \texttt{pmd}$(M)$. 

\begin{lemma}{(\textbf{Correctness})}
    The output of \cref{alg:total-decomposition} is an interval basis.
\end{lemma}
\begin{proof}
The set $\mathscr{V}$ from \cref{alg:total-decomposition} is the union of sets $\mathcal{V}^i_\text{Birth}$ from \cref{alg:ssd-general}, so it matches the definition of $\mathscr{V}$ given in Definition \ref{def:adapted_basis_set}. Then correctness follows from \cref{th:correctness}.
\end{proof}

\begin{example}
Consider the $\R$-persistence module specialized from \cref{ex:RunningExPersMod}.

\[ 0 \ \underset{\scalebox{0.6}{$\begin{pmatrix} 0 
\end{pmatrix}$}}{\overset{\varphi_0}{\longrightarrow}} \ \R \ \underset{\scalebox{0.6}{$\begin{pmatrix} 1 \\ 0 \end{pmatrix}$}}{\overset{\varphi_1}{\longrightarrow}} \ \R^2 \ \underset{\scalebox{0.6}{$\begin{pmatrix} 1 & 1 \end{pmatrix}$}}{\overset{\varphi_2}{\longrightarrow}} \ \R \ \underset{\scalebox{0.6}{$\begin{pmatrix}  0 \end{pmatrix}$}}{\overset{\varphi_3}{\longrightarrow}} \ 0. \]

 We showcase the procedure of \cref{alg:total-decomposition} and compute its interval basis. This example matches the persistence module generated by persistent homology in \cref{fig:example_classical}.
 For $i=0,1,2,3$ we need to compute $\mathcal{V}^i_{\text{Birth}}$. 
 Notice that $\varphi_0$ is the null map, so the flag for the first step is trivial and $\mathcal{V}^0_{\text{Birth}}$ is empty. 
 
 For $i=1$, we have $\im(\varphi_{i-1}) = 0$ and $\ker (\varphi_{1,2}) = \ker (\varphi_{1,3}) = 0$, so $\R = \ker (\varphi_{1,4})$. By \texttt{ssd}, we extend a basis of $\im (\varphi_0)$ (which is empty) to a basis of $\R$, which yields vector $1$. Then $\mathcal{V}^1_{\text{Birth}} = \{ 1\}$ with persistence pair $(1,4)$. 
 
 For $i=2$, we have $\im(\varphi_{i-1}) = \langle${$\left(\begin{smallmatrix} 1 \\ 0 \end{smallmatrix}\right)$}$\rangle$. Furthermore $\ker (\varphi_{2,3}) = \langle${$\left(\begin{smallmatrix} 1 \\ -1 \end{smallmatrix}\right)$}$ \rangle$, so we extend the basis of $\im (\varphi_1)$ against the basis of $\ker (\varphi_{2,3})$ obtaining set $\{
 \left(\begin{smallmatrix} 1 \\ -1 \end{smallmatrix}\right),
 \left(\begin{smallmatrix} 1 \\ 0 \end{smallmatrix}\right)
 \}$, which spans $\R^2$, so \texttt{ssd} terminates setting $\mathcal{V}^2_{\text{Birth}} = \{ ${$\left(\begin{smallmatrix} 1 \\ -1 \end{smallmatrix}\right)$}$\}$ with persistence pair $(2,3)$. 
 
 For $i=3$, we have $\im(\varphi_{i-1}) = \R$, so $\mathcal{V}^3_{\text{Birth}}$ is empty. 

\noindent
Finally, the interval basis is $ \mathscr{V} = \{ 1 ,$ {$\left(\begin{smallmatrix} 1 \\ -1 \end{smallmatrix}\right)$} $\}$, with persistence diagram $PD = \{ (1,4),(2,3)\}$. It is (up to the irrelevant sign) the same result as in the example in \cref{ex:RunningExPersMod}, as vector $1$ in degree $1$ corresponds to the first generator $v_1$, and vector {$\left(\begin{smallmatrix} 1 \\ -1 \end{smallmatrix}\right)$} in degree $2$ corresponds to the difference of the second and third $v_2 - v_3 = xv_1 - v_3$. 
\end{example}

\section{Computational complexity}
\label{sec:complexity}

In this section, we estimate the computational complexity of the parallel algorithm for decomposing a persistence module just presented in~\cref{sec:Decomposition}.
The evaluation of our parallel~\cref{alg:total-decomposition} depends on the output barcode in terms of the number of intervals and their length. In the final part, we discuss the ``unbalanced'' case, when some steps are significantly more complex than others. We argue that, in the worst-case, our parallel algorithm is as expensive as the known procedure of the graded Smith normal form, here specialized to a persistence module presentation matrix as described in~\cite{corbet2018}. \\

Let us assume that our persistence module has steps $M_i$, each having dimension $m_i$ for $i$ varying in $\{ 0,\dots, n+1 \}$, where $m_0=m_{n+1}=0$,  $m=\sum_i m_i$, and $\bar{m}=\max_{i} m_i$.
We assume a parallel implementation of \cref{alg:total-decomposition}.
Hence, we focus on the single-step decomposition performed by~\cref{alg:ssd-general} on each step $i=1,\dots, n$. 
First, a \texttt{ColumnReduction} (\cref{alg:col-red}) is called only once, before entering the outer for-loop, to extract the image of $\varphi_{i-1}$, which reduces a matrix of size $m_{i}\times m_{i-1}$. 
We observe that, inside the inner for-loop, the total number of operations depends on the parameter $k_i=m_i - r_i$, where $r_i=\rank (\varphi_{i-1})$, and on the variable parameter $r_s$ that counts the number of columns that have not yet been reduced to zero. 
We claim to estimate the time complexity of \cref{alg:total-decomposition} in parallel as 
\begin{equation}
O(\bar{m}^2V_i),
    \label{eq:output-complexity}
\end{equation}
 where we define the {\em output-dependent} parameter $V_i = \sum_s r_s$. \\

Indeed, within the inner for-loop, for each $s=0, \dots, n-i$:
\begin{itemize}
    \item a matrix multiplication is called for matrices of size $m_{i+1+s}\times m_{i+s}$ and $m_{i+s}\times r_s$;
    \item a column reduction (\texttt{ColumnReduction}) is called for a matrix of size\\ $m_{i+1+s} \times r_s$;
    \item a basis completion (\texttt{bca}), \cref{alg:basis-extension}, is called for a list of at most $ m_i-r_s$ vectors and a list of $\vert \text{newInds} \vert$ vectors.
\end{itemize}

The complexity of the calls to the \texttt{bca} subroutine depends on parameter $s$. The subroutine is performed in chunks as $s$ increases, but the sum of the $\vert \text{newInds} \vert$ eventually amounts to $k_i$. The worst case happens when $r_s$ decreases one by one as $s$ increases. In this case, the cost of \texttt{bca} for each $s$ amounts to that of reducing one column of length $m_i$ against a list of $r_i+j$ with $j=1, \dots, k_i-1$ (increasing with $s$).  
The total cost over the for-loop is therefore $O(m_ir_ik_i + m_ik_i^2)$, where parameters are related via $k_i = m_i - r_i $. 
By substitution one then obtains that the cost of the \texttt{bca} subroutine is bounded by $O(m_i^3)$.

Let us now consider the other two steps. The cost of matrix multiplication between a $m_{i+1+s}\times m_{i+s}$ and a $m_{i+s}\times r_s$ matrix is $O(m_{i+1+s} \ m_{i+s} \ r_s)$. Let us consider the worst case $\bar{m}$ for all $m_i$'s. We get $O(\bar{m}^2r_s)$ for each step $s$. 
The cost of column reduction of a matrix of size $m_{i+1+s} \times r_s$ can also be bounded by $O(\bar{m}^2r_s)$.
Now, the parameter $V_i= \sum_s r_s$ is, intuitively, the ``volume'' of all bars born at step $M_i$ until their death, and we can express the total cost of matrix multiplication and column reduction by $O(\bar{m}^2V_i)$. This makes the contribution of \texttt{bca} negligible and hence provides the global cost of \Cref{alg:ssd-general}. \\

We remark for completeness that the parallel execution of the steps of \cref{alg:total-decomposition} may, in particular cases, be very unbalanced in terms of time complexity. We therefore also provide an input-dependent estimate of the worst-case time complexity of \cref{alg:total-decomposition}, in terms of parameters $m$ and $\bar{m}$. This is then compared to that of the graded SNF algorithm \cite{skraba+2013x} (described here as \cref{alg:graded-SNF} in \Cref{sec:algo-graded-normal}).

In the case of our parallel \cref{alg:total-decomposition}, we have $V_i\leq (n-i+1)\bar{m}$. The equality corresponds to a ``rectangular'' barcode where all interval modules started at step $i$ are non-trivial until step $i=n$. The single-step decomposition  attains its worst complexity when the rectangular barcode starts at the first step $i=1$.
In that case, we can change the output-dependent estimate $O(\bar{m}^2V_1)$ into $O(\bar{m}^3n)$.
Observe that parameter $m$, expressing the sum of all $m_i$'s, is now equal to $\bar{m} n$. Hence, we obtain the {\em input-dependent} estimate for the most unbalanced case:
\begin{equation}
O(\bar{m}^2 m).
    \label{eq:input-complexity}
\end{equation}

Let us now focus on the graded SNF algorithm (\Cref{alg:graded-SNF} in \Cref{sec:gradedSmith}, when applied to presentation matrices whose block structure is described in \Cref{sec:from-pers-module-to-matrix}); that is, when we assume the input matrix $S$ to be of size $m\times m$, subdivided into $n$ blocks of size $ m_i \times (m_i+m_{i+1})$ with generators (rows) of degree $i$, and relations (columns) of degree $i$ and $i+1$.
In principle, storing an $m \times m$ matrix incurs a space cost of $O(m^2)$ to represent the input, which is significantly higher than our parallel approach $O(\bar{m}^2)$. However, a sparse implementation of matrix $S$ can make the two space complexities equivalent.
As for the time complexity, the classical estimate for SNF reduction is $O(m^3)$. This can theoretically be reduced via optimized algorithms to $O(m^{\omega})$ (see \cite{Storjohann1996}), where $2\leq \omega <3$ is the lower bound for the complexity of matrix multiplication.
However, this can also be significantly reduced by taking into account the block structure of $S$ in our particular case. 
Indeed, for each of the $m$ processed columns (see \Cref{alg:graded-SNF} in \Cref{sec:algo-graded-normal}) 
the procedure performs theoretically at most $m$ row reductions. However, in the input matrix the non-trivial entries for each column of degree $i$ straddle at most two blocks of row degree $i-1$ and $i$.
For each column of degree $i$, the row operations performed previously
can produce non-trivial entries from degree $1$ to $i$, thus breaking the block structure vertically. 
However, the number of non-trivial entries in the column is still $O(\bar{m})$ since 
$m_0+\dots+m_{i-1}$ pivots have already been found, and their rows been reduced to zero. These rows can lie no lower than those of degree $i$, leaving room for no more than $O(\bar{m})$ non-trivial entries in the column. 
%
%
Hence, the first inner for-loop (on index $j$ in \Cref{alg:graded-SNF}) performs $O(\bar{m})$ operations.
The second inner for-loop (on index $c$ in \Cref{alg:graded-SNF}) simply sets to zero the $O(\bar{m})$ non-trivial entries of the row, thus its cost is negligible.
This means that the cost of the algorithm on each block $m_i \times (m_i+m_{i+1})$ is $O(\bar{m}^3)$, hence the overall cost is $O(n\bar{m}^3)$.
By substituting $m=\bar{m}n$, we conclude that the worst-case cost of the parallel approach in \labelcref{eq:input-complexity} is the same as the cost of the graded Smith normal form specialized to the block structure of a persistence module presentation.

\section{Persistent homology modules} \label{sec:PersistentHomology}

In this section, we provide a parallel method to obtain a persistence module by applying the $\tth{k}$-homology functor to an equi-oriented finite sequence of chain complex maps that are not necessarily injective.
First, we fix the notation.
Then, we show a construction of the $\tth{k}$-persistent homology module by homology representatives, then by harmonic representatives obtained through the combinatorial Laplacian operator. 
For the general homology representative case, our approach is a simple adaptation of known algorithms independently acting on each step in the input sequence of chain complexes.
Here, we state the desired properties to be fulfilled by the chosen method.

We aim to underline the level of generality of our approach and to exemplify the possibility of being adaptable to a large variety of special homology representatives. 
We observe that the parallel approach may lead to unnecessary computation repetitions in the case of injective chain maps, namely in the case of persistent homology. 
The case of the harmonics instead admits a more efficient construction of the module via simple matrix multiplications thanks to our proof of~\cref{thm:iso-hodge-hom}. \\

A {\em chain complex} with coefficients in $\mathbb{F}$ is a sequence $C=(C_\bullet,\partial_\bullet)$ of $\mathbb{F}$-vector spaces connected by linear maps with $k\in\mathbb{N}$

\[ \dots \overset{\partial_{k+2}}{\longrightarrow} C_{k+1}\overset{\partial_{k+1}}{\longrightarrow}  C_k \overset{\partial_{k}}{\longrightarrow} C_{k-1} \overset{\partial_{k-1}}{\longrightarrow} \dots \overset{\partial_{2}}{\longrightarrow} C_1 \overset{\partial_{1}}{\longrightarrow} C_0 \overset{\partial_0}{\longrightarrow} 0, \]

\noindent
such that $\partial_{k+1}\partial_k=0$ for all $k\in\mathbb{N}$.
Each vector space $C_k$ is called the space of $k${\em -chains}.
The subspace $Z_k=\ker(\partial_k)$ is called  the space of $k${\em -cycles}.
The subspace $B_k=\im(\partial_{k+1})$ is called  the space of $k${\em -boundaries}. The condition $\partial_{k+1}\partial_k=0$ ensures that $B_k\subseteq Z_k$, for all $ k\in\mathbb{N}$. The quotient space $H_k=Z_k/B_k$ is the $k${\em -homology space}. 
A {\em chain map} $f : (C_\bullet,\partial^C_\bullet) \longrightarrow (D_\bullet,\partial^D_\bullet)$ is a collection of linear maps $f_k:C_k \longrightarrow D_k$ such that $f_k\partial^C_{k+1} = \partial^D_{k+1}f_{k+1}, \mbox{ for all } k\in\mathbb{N}.$

A chain map induces linear maps $\widetilde{f}_k: H^C_k \longrightarrow H^D_k$, for all $k$, and it can be shown that this implies that the $H_k$ are indeed functors from the category of chain complexes and chain maps to the category of vector spaces over $\F$ and linear mappings (see \cite{Hatcher2002book} for a complete account of these facts).

To compute the matrix associated with the map $\widetilde{f}_k: H^C_k \longrightarrow H^D_k$, we assume to have a basis $\{h^C_1,\dots,h^C_{\beta^C_k}, b^C_1, \dots, b^C_q\}$ of $Z_k(C)$, where $\{b^C_1, \dots, b^C_q\}$ is a basis of $B_k(C)$, and basis $\{h^D_1,\dots,h^D_{\beta^D_k}, b^D_1, \dots, b^D_r\}$ of $Z_k(D)$, such that  
$\{b^D_1,\dots,b^D_r\}$ is a basis of $B_k(D)$. Such basis can be found applying \cref{alg:homology} in ~\cref{subsec:homology-algorithms}.
Then, for each $s=1, \dots, \beta_k^C$, the following linear system in the variables $\lambda_1^s, \dots, \lambda_{\beta_k^D}^s, \mu_1^s, \dots, \mu_r^s$ must be solved:
\begin{equation}
f_k(h_s^C) = \sum_{j=1}^{\beta_k^D}\lambda_j^s h_j^D + \sum_{l=1}^r\mu_l^s b_l,
    \label{eq:linear-systems-structure-maps}
\end{equation}
defining the matrix with columns $(\lambda^s_1,\dots, \lambda^s_{\beta_k^D})^T$, with  $s=1, \dots, \beta_k^C$, as the matrix of $\widetilde{f}_k$ with respect to the basis induced by the projection of $\{h^C_1,\dots,h^C_{\beta^C_k}\}$ and $\{h^D_1,\dots,h^D_{\beta^D_k}\}$ to their respective homology space (see \cref{alg:general-induced-map} in~\cref{subsec:homology-algorithms}).

The application of the functor $H_k$ to a given sequence of complexes and chain maps 
\begin{equation}
\xymatrix{
\, C_{\bullet}^1 \ar[r]^-{f^1} & \dots \ar[r] & \dots \ar[r]^-{f^{i-1}} & C_{\bullet}^i \ar[r]^-{f^{i}} & \dots \ar[r]^-{f^{n-1}} & C_{\bullet}^n
}
\label{def:chain-filtration}
\end{equation}
provides a persistence module $\{(H_k^i,\widetilde{f}_{k}^i)\}_{i=0}^n$ for all $k\geq 0$. This persistence module is called the $\tth{k}$-persistent homology of the sequence of chain complexes \labelcref{def:chain-filtration}, see \cite{edelsbrunner2008survey}.

\subsection{Parallel construction of the persistent homology module via harmonics}
\label{subsec:harmonics}

In this section, we describe a parallel construction of the persistence module $\{(\mathcal{H}_k^i,\hat{f}_i)\}_{i\in\mathbb{N}}$ where $\mathcal{H}_k^i$ is the space of $k${\em -harmonics} at step $i$ and coefficients are taken in $\mathbb{R}$.
We call the persistence module $\{(\mathcal{H}_k^i,\hat{f}_i)\}_{i\in\mathbb{N}}$ the  {\em $\tth{k}$-harmonic persistence module}.

After some preliminaries on the Hodge Laplacian operator, through the Hodge decomposition (\cref{thm:hodge}, \cref{thm:hodge-map}),
for each index $i\in\mathbb{N}$, we show that there exists a structure map $\hat{f}_i$ induced by $f_k$, such that the $\tth{k}$-persistent homology module $\{(H_k^i,\widetilde{f}_i)\}$ and the $\tth{k}$-harmonic persistence module $\{(\mathcal{H}^i_k,\hat{f}_i)\}$ are isomorphic. We then provide \cref{alg:laplacian-induced-map} to compute these maps. 

\subsubsection*{The Hodge Laplacian}
In this section, we fix $\mathbb{F} = \mathbb{R}$. Given a chain complex $(C_\bullet,\partial^C_\bullet)$, we choose an inner product  $\langle \cdot , \cdot \rangle_k$ on each space of $C_k$ so that we have a well-defined adjoint of $\partial_k$, i.e. the map $\partial_k^*:C_{k-1}\rightarrow C_k$ such that $ \langle \partial_{k}(c),d \rangle_{k-1} = \langle c,\partial^*_{k}(d) \rangle_k $, for all $c\in C_{k}, d\in C_{k-1}$.

For $k\in\mathbb{N}$, the {\em Hodge Laplacian} in degree $k$  (Laplacian, for short) is the linear map on $k$-chains $L_k:C_k \longrightarrow C_k$ given by

\begin{equation}
    L_k := \partial_{k+1} \partial^*_{k+1} + \partial^*_{k} \partial_{k}.
    \label{def:laplacian}
\end{equation}

The space of $k${-harmonics} of a chain complex is the subspace of $C_k$

\begin{equation}
    \mathcal{H}_k := \ker(L_k).
    \label{def:harmonics}
\end{equation}

We refer to~\cite{Horak+2013} for more details and
we recall the following theorems, see Section 5.1 of \cite{Lim2015HodgeGraphs}.
\begin{theorem}\label{thm:hodge}

For a chain complex $C$ and for every natural $k$, 
    \[C_k = \mathcal{H}_k \oplus \im (\partial_{k+1}) \oplus \im (\partial_k^*)\]
    \noindent
    Moreover, this decomposition is orthogonal and $Z_k = \mathcal{H}_k \oplus \im(\partial_{k+1})$.
\end{theorem}

\begin{theorem}\label{thm:hodge-map}
The linear map $\psi_k: \mathcal{H}_k \longrightarrow H_k$ defined by $\psi_k(h)=[h]$ is an isomorphism, where $[h]$ is the homology class induced by the cycle $h$.

\end{theorem}

An obstacle to the persistence of the harmonic space is the following:
\begin{remark}
    A chain map $f:C\longrightarrow D$ does not restrict to a map between the harmonic subspaces $\mathcal{H}_k^C$ and $\mathcal{H}_k^D$.
    \label{rem:harmonics-not functorial}
\end{remark}
Indeed, given an element $h\in\mathcal{H}_k^C$, the $k$-cycle $f(h)$ is not necessarily in $\mathcal{H}_k^D$. 
More precisely, $f(h)$ is necessarily a $k$-cycle but not necessarily a $k$-cocycle.

However, given a sequence of chain complexes and chain maps as in \labelcref{def:chain-filtration}, we want to construct a persistence module, isomorphic to the persistent homology module of the sequence, given by the harmonic spaces of the chain complexes, with maps induced by the chain maps. 
The following theorem is sufficient to provide such a persistence module.

\begin{theorem}
    For any chain map $f:C\longrightarrow D$ and any $k \in \mathbb{N}$, the following diagram commutes

\[
\begin{tikzcd}
\mathcal{H}^C_k\arrow[r, "\hat{f}_k"]\arrow[d, "\psi_k^C"] & \mathcal{H}^D_k\arrow[d, "\psi_k^{D}"]\\
H^C_k \arrow[r, "\widetilde{f}_k"] & H^D_k.                         
\end{tikzcd}
\]
where $\hat{f}_k = \pi_k^Df_k i_k^C$, $i_k^C$ is the natural inclusion of $\mathcal{H}^C_k$ into $C_k$ and $\pi_k^D$ is the orthogonal projection of $D_k$ onto $\mathcal{H}^D_k$.
\label{thm:iso-hodge-hom}
\end{theorem}

\begin{proof}
For any $h\in \mathcal{H}_k^C$, we can see that $\widetilde{f}_k(\psi_k^C(h)) = \psi_k^D(\hat{f}_k(h))$. In fact, by the definition of $\widetilde{f}$ and $\hat{f}$, it holds $\widetilde{f}_k(\psi_k^C(h)) = \widetilde{f}_k([h]_C) = \left[f_k(h)\right]_D$ and $\psi_k^D(\hat{f}_k(h)) = \left[\pi_k^D(f_k(h))\right]_D$. Since $f_k(h)$ is a cycle in $D_k$ and because of the decomposition in \cref{thm:hodge}, there is a boundary $b$ of $D_k$ such that $f_k(h) = \pi_k^D(f_k(h)) + b$, hence $\left[f_k(h)\right]_D = \left[\pi_k^D(f_k(h))\right]_D$ and the diagram commutes.
\end{proof}

The $\hat{f}_k$ matrix can then be easily computed with \cref{alg:laplacian-induced-map}.

\begin{algorithm}
\SetAlgoLined
\textbf{Input}: Chain map $f_k:C_k(C)\to C_k(D)$, $\{v_1^C, \dots, v_n^C\}$ orthonormal basis of $\mathcal{H}_k^C$, $\{w_1^D, \dots, w_m^D\}$ orthonormal basis of $\mathcal{H}_k^D$ \; 
\KwResult{Matrix $\Phi$ representing $\hat{f}_k:\mathcal{H}_k(C)\to \mathcal{H}_k(D)$} 
$V_C:=$ matrix with columns $v_1^C, \dots, v_n^C$\;
$V_D:=$ matrix with columns $w_1^D, \dots, w_m^D$\;
$\Phi = V_D^T\cdot f_k \cdot V_C$\; 
\Return{ $(\Phi_i)$}
 \caption{Induced map between Laplacian kernels}
 \label{alg:laplacian-induced-map}
\end{algorithm}

\subsection{Simplicial complex chains}
\label{subsec:simplicial}

An {\em (abstract) simplicial complex} $\Sigma$ on a finite set $V$ is a subset of the power set of $V$, with the property of being closed under restriction. An element of $\Sigma$ is called a {\em simplex} and if $\sigma \in \Sigma, \ \tau \subseteq \sigma$ then $\tau \in \Sigma$. Elements of $V$ are usually called {\em vertices}. Simplices of cardinality $k+1$ are called $k$-simplices. We also say a $k$-simplex has {\em dimension} $k$. We call the $k$-skeleton of $\Sigma$ the set of simplices of $\Sigma$ of dimension $\leq k$, denoted $\Sigma_k$. If $\tau \subseteq \sigma$ we say that $\tau$ is a face of $\sigma$ and $\sigma$ is a coface of $\tau$. The dimension of a simplicial complex is defined as $\dim \Sigma := \max \{ \dim \sigma \ | \ \sigma \in \Sigma \}$.
By numbering the vertices in $V$, we define a {\em positvely oriented} $k$-simplex $\sigma = [v_0,\dots, v_k]$  as the class of tuples $( v_{p(0)},\dots, v_{p(k)} )$ with $p$ an even permutation. All remaining permutations give the {\em negatively oriented } simplex $\sigma$.

It is possible to specialize the chain complex construction of \cref{sec:PersistentHomology} to the case of a simplicial complex 
$\Sigma$, with coefficients in a field $\F$.
We obtain a chain complex $(C_\bullet,\partial_\bullet)$ by defining, for each $k$, $C_k = C_k(\Sigma)$, where $C_k(\Sigma)$ is the space of $k${\em -simplicial chains} consisting of finite $\F$-linear combinations of the oriented $k$-simplices of $\Sigma$ and such that $-\sigma$ coincides with the opposite orientation on $\sigma$.
We define $\partial_k = \partial_k(\Sigma)$, where $\partial_k(\Sigma):C_k(\Sigma) \rightarrow C_{k-1}(\Sigma)$ is the {\em simplicial boundary map} defined on an element of the canonical basis $\sigma=[v_0, \dots, v_k]\in\Sigma_k$ by
$
\partial_k(\sigma)=\sum_{i=0}^k (-1)^i [v_0, \dots, \hat{v}_i, \dots, v_k],
$
where $\hat{v}_i$ means that vertex $v_i$ is omitted. It extends to the whole chain space by linearity. 

A {\em simplicial map} $s:\Sigma \rightarrow \Sigma'$ is the extension to $\Sigma$ of a vertex map $\tilde{s} : \Sigma_0 \rightarrow \Sigma'_0$, given by $s(\sigma) = [ \tilde{s}(v_0) , \dots, \tilde{s}(v_k) ]$ whenever $\sigma = [v_0, \dots, v_k] \in \Sigma$.
Denote by $(C_\bullet,\partial^C_\bullet)$ the simplicial chain complex of $\Sigma$ and by $(D_\bullet,\partial^D_\bullet)$ the simplicial chain complex of $\Sigma'$.
Then the simplicial map $s$ induces a chain map $f : (C_\bullet,\partial^C_\bullet) \rightarrow (D_\bullet,\partial^D_\bullet)$ by setting $f_k(\sigma) = s(\sigma)$ for all $\sigma \in \Sigma_k$ if $\dim s(\sigma) = k$, and $0$ otherwise.

We can apply this section's persistent homology module constructions to any monotone sequence of simplicial maps. Hence, our parallel algorithm introduced in \cref{sec:Decomposition} specializes in the simplicial complex case.
In the following, we report some examples and discuss interesting points concerning the simplicial case.

\begin{figure}
    \centering
    \includegraphics[width = 0.5\textwidth]{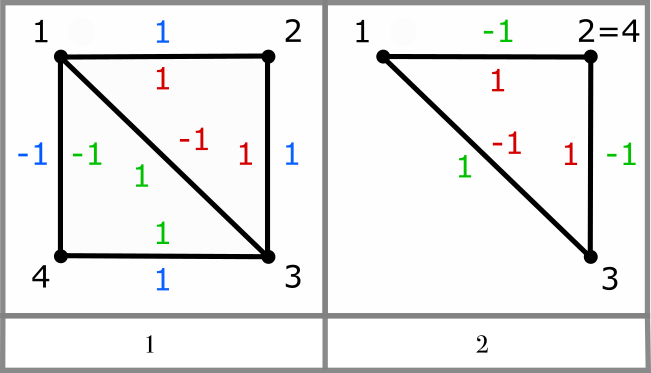}
    \caption{An example of a vertex collapse inducing a chain map $f=(f_k)$, where step 2 is obtained from step 1 by identifying vertices $2$ and $4$. Coefficients of $1$-chains of possible degree-1 homology representatives are depicted with the same color. The figure shows, at step 1 in red $z_1 = [1,2] + [2,3] - [1,3] $, in green $z_2 = [1,3] + [3,4] - [1,4]$, in blue $z = [1,2] + [2,3] + [3,4] - [1,4]$; at step 2 in red $\omega_1 = [1,2] + [2,3] - [1,3]$, in green $\omega_2 = [1,3] - [2,3] - [1,2]$. The $1$-component $f_1$ of the chain map $f$ sends $z_1$ to $\omega_1$, and $z_2$ to $\omega_2$. Hence, the image of $z=z_1+z_2$ is trivial. 
    }
    \label{fig:motivating_vertex_collapse}
\end{figure}

\begin{example}[Vertex collapse]
Consider the chain map $f$ induced by the vertex collapse in~\cref{fig:motivating_vertex_collapse}.
We apply the above construction to retrieving the associated 2-step persistent homology module in degree 1.
Depending on the chosen vertex labeling and optimization procedure, our parallel construction of the homology steps might return several choices of homology representatives.
For instance, according to the labeling of vertexes in black, the already mentioned left-to-right reduction in~\cite{Carlsson2005DCG} would return the red and green $1$-chains at step 1 and the red $1$-chain at step 2.
Furthermore, solving the linear systems in \cref{eq:linear-systems-structure-maps} yields matrix $\begin{pmatrix}
1, & -1
\end{pmatrix}$, representing the linear map $\widetilde{f}$, meaning that the green and red homology representatives from step 1 are mapped to the red homology representative of step 2 with the opposite sign. This concludes the construction of the desired 2-step persistence module. 
We observe that the red and green homology representatives do not form an interval basis since, at step 2, they are non-trivial and equivalent to one another.
The parallel decomposition previously introduced in~\cref{sec:Decomposition} can be applied to the obtained persistence module to get the interval basis formed by the red and blue homology representatives at step 1. This way, the blue representative captures the homology class (the sum of the red and green representatives) being born at step 1 and dying at step 2, and the red one captures the class being born at step 1 and still non-trivial at step 2.
    \label{eg:vertex-collapse}
\end{example}

\begin{figure}
    \centering
    \includegraphics[width = 0.5\textwidth]{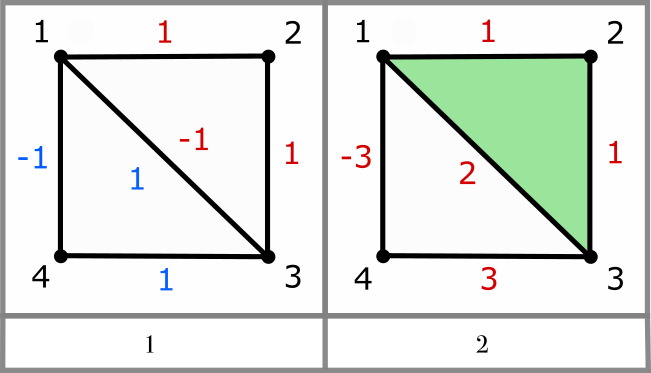}
    \caption{An example of tracking of harmonic homology representatives along the inclusion of simplicial complexes, obtained by inserting the $2$-simplex $[1,2,3]$. 
    Coefficients of $1$-chains forming harmonic homology representatives for the associated 2-step persistence module are depicted with the same color in each step. The figure shows, at step 1 in red $z_1 = [1,2] + [2,3] - [1,3] $, in blue $z_2 = [1,3] + [3,4] - [1,4]$; at step 2 in red $\omega = [1,2]+2 [1,3] -3 [1,4] + [2,3] +3 [3,4]$.
    The $1$-component $f_1$ of the chain map $f$ sends $z_1 + 3 z_2$ to $\omega$.
    }
    \label{fig:motivating_harmoincs_representatives}
\end{figure}

\begin{example}[Tracking harmonic homology representatives]
Consider the chain map $f$ induced by the inclusion in \cref{fig:motivating_harmoincs_representatives}.
We apply the construction above to compute harmonic homology representatives in the first degree for the two steps: red and blue on the left (step 1) and red on the right (step 2).
In step 1, we notice that the two obtained harmonics coincide with generic homology representatives. This is due to the absence of $2$-simplices.
However, harmonic representatives are not, in general, preserved by the inclusion of simplicial complexes.
Indeed, in step 2, the same homology representatives are no longer harmonic forms. 
By~\cref{alg:laplacian-induced-map}, we retrieve that the red $1$-chain $\omega$ at step 2 is a combination of the inclusions of the red $z_1$ and blue $z_2$ $1$-chains at step 1: $\omega = z_1 + 3 z_2$.
This concludes constructing the desired $2$-step persistent homology module via harmonic representatives.
We observe that $z_1$ and $z_2$ do not form an interval basis for the obtained persistence module.
Our parallel decomposition allows us to choose $z = z_2 - 3 z_1$ and $z_2$ as harmonic representatives at step 1 so that the inclusion directly maps $z$ to $\omega$ and $z_2$ to $0$ and each harmonic representative is kept independent from other representatives along the module steps.
    \label{eg:harmonics-inclusione}
\end{example}

\subsubsection{Filtered simplicial complexes}
\label{subsec:filtered}

A simplicial complex $\Sigma$ can be made into a {\em filtered simplicial complex} (or a {\em filtration of simplicial complexes}) by taking a finite sequence of subcomplexes $\Sigma^0\subseteq \Sigma^1 \subseteq \dots \subseteq \Sigma^n = \Sigma$, where a {\em subcomplex} is a subset and also a simplicial complex.
The inclusion maps in the filtration induce a monotone sequence of chain complexes with maps $\{f^i\}$ as in \labelcref{def:chain-filtration} where the induced chain maps are all injective.
We fix a dimension $k$.
We get that the filtered chain complex $C_k = \{ (C_k^i, f^i_k) \}_{i=1}^{n}$ is a persistence module.

\begin{remark}
The $\F[x]$-graded module $\alpha(C_k)$ associated with the persistence module of the filtered chain complex $C_k$ is free.
Moreover, an interval basis for $C_k$ consists, for each index $i$, of the $k$-simplices $\sigma$ in $\Sigma^i \setminus \Sigma^{i-1}$.
    \label{rem:chains-free-pers-module}
\end{remark}

The freeness of $\alpha(C_k)$ implies that also the graded modules associated with the persistence modules of the {\em filtered $k$-cycles}  $Z_k$ and the persistence modules of the {\em filtered $k$-boundaries} $B_k$ are free.
To this purpose, we underline the following observations concerning the left-to-right reduction $\partial = RV$~\cite{Carlsson2005DCG} of the boundary matrix $\partial$ which is at the heart of most of the persistent homology computations, where $R$ is the reduced matrix and $V$ keeps track of the operations performed on the columns along the reduction.

\begin{remark}
The collection $\mathscr{V}$ of the homology representatives in the columns of $V$ corresponding to the null columns in $R$ form an interval basis for $Z_k$.
The same collection $\mathscr{V}$ is a minimal system of generators of the persistent homology module $H_k$, which is not generally an interval basis. A presentation of $H_k$ requires further column reductions to find the combinations expressing a system of generators of $B_k$ in terms of $\mathscr{V}$.
An example of such a system of generators $\mathscr{V}$ is in \cref{fig:example_classical} in \cref{sec:intro}.
    \label{rem:to-reviewer-1}
\end{remark}

\begin{remark}
The collection $\mathscr{V}$ of the homology representatives in the non-trivial columns of $R$ form an interval basis for $B_k$.
If the graded module associated with the persistent homology module $H_k$ has no free part, the same collection $\mathscr{V}$ is also an interval basis for $H_k$.
Otherwise, the same collection $\mathscr{V}$ can be extended to an interval basis of $H_k$ by adding the elements of the interval basis of $Z_k$ of infinite order in the associated graded module (essential classes).
An example of such an interval basis is in \cref{fig:example_interval} in \cref{sec:intro}. The retrieval of this kind of homology representatives is implemented, for instance, in \cite{Bauer2016ripser}, by the clear optimization procedure~\cite{Bauer2014}.
    \label{rem:to-reviewer-2}
\end{remark}

Indeed, one can easily check that for $Z_k$; we get a minimal system of generators $\mathscr{V}_Z$ in the case of Remark \ref{rem:to-reviewer-1}.
By simply noticing that boundaries are cycles, the system $\mathscr{V}_B$ in the case Remark \ref{rem:to-reviewer-2} is a minimal system of generators of $Z_k$.
Hence, both $\mathscr{V}_Z$ and $\mathscr{V}_B$ are a minimal system of generators for $H_k$.

By freeness, the two systems form an interval basis of $Z_k$.
Moreover, $\mathscr{V}_B$ contains a minimal system of generators of $B_k$, and hence an interval basis for $B_k$.
Notice that the reduction provides a presentation matrix of $H_k$ with respect to generators in $\mathscr{V}_B$, whereas a presentation with respect to $\mathscr{V}_Z$ is less direct to retrieve.
For $\mathscr{V}_B$, the associated presentation matrix has one row per element $\mathscr{V}_B$ and one column per element of $\mathscr{V}_B$ belonging to the interval basis of $B_k$. Rows are graded as elements of $Z_k$. Columns are graded as elements of $B_K$.

The subset in $\mathscr{V}_B$, which forms the interval basis for $B_k$, defines the rows containing exactly a non-trivial element in correspondence with the column of the matched boundary.
Indeed, if a generator $v$ of degree $i$ in $\mathscr{V}_B$ is matched to $w$, a generator of degree $j\geq i$ of $B_k$, this implies that $x^{j-i} v = 0$. All other entries are trivial.
Hence, the presentation matrix with respect to $\mathscr{V}_B$ is in Smith normal form according to Definition \ref{def:snf-matrix} and hence $\mathscr{V}_B$ is an interval basis of $H_k$. 
For $\mathscr{V}_Z$, if a generator $v$ of degree $i$ in $\mathscr{V}_Z$ is matched to $w$ a generator of degree $j\geq i$ of $B_k$, this does not imply that $x^{j-i} v = 0$ since there could be another element $v'$ of degree $h\leq i$ such that $x^{j-i} v = x^{j-h} v'\neq 0$. 

\section{Conclusions and future works}
\label{sec:conclusions}

In this work, we have described an interval basis of a persistence module as a particular minimal system of generators.
We have introduced \cref{alg:total-decomposition} as a distributed approach for retrieving an interval basis. Our approach applies to any persistence module as defined in~\cref{sec:Background}, not necessarily coming from the homology of a filtered chain complex.
A specialization for real coefficients based on the SVD decomposition is also available in \cref{subsec:real}. 

In \cref{sec:complexity}, we have discussed the computational advantage of our parallel approach, quantified through an output-dependent estimate, including a focus on the unbalanced barcode case (the worst case), with all bars appearing at the same step, especially at the very first step, and lasting for along the entire persistence module. We have discussed the complexity of the graded Smith normal form reduction of persistence module presentation matrices, showing that the complexity of the unbalanced parallel case matches the cost of the specialized SNF reduction. \\

In TDA applications, often dealing with persistence modules obtained through the homology functor in degree $0$ of a geometric filtration (e.g. Vietoris-Rips, Alpha, etc), such an unbalanced configuration may arise, in that all degree-zero bars appear at the first step.
Even though one typically has that most of the homology classes soon disappear along the filtration, it is fair to highlight that our parallel approach would have better performances in homology degrees other than $0$, due to the higher sparsity of the obtained barcode. Furthermore, our approach could be more advantageous in more general TDA frameworks, such as in the case of non-injective simplicial chain maps, or the case of harmonic forms treated in \cref{subsec:harmonics}. \\

We have described how to obtain a persistent homology module out of a monotone sequence of chain complexes, remarking that each step and structure map in the module can be obtained independently, thus making it suitable for parallel approaches.
Such an integration has offered interesting insights to be investigated further. For instance, it has made it possible to geometrically locate the interval basis vectors onto a filtered simplicial complex.
We have discussed simple examples to make comparisons with two possible kinds of homology representatives obtained through the reduction algorithm \cite{Carlsson2005DCG}, and discussed which kind of homology representatives do satisfy the interval basis definition in Remark \ref{rem:to-reviewer-1} and Remark \ref{rem:to-reviewer-2}.

We believe that, for a monotone sequence of chain maps, the interval basis's descriptive power deserves further inquiry, since it encodes implicitly the relations among evolving homology classes. Possible future directions on the descriptive power of interval bases include adapting our parallel approach to retrieving interval bases of submodules. This might, in our opinion, apply to the challenge of defining interval matchings induced by persistence module morphisms, to be compared with the ones already available in the literature.
Other directions may include the study of interval bases applied to multiparameter persistent homology and to non-injective families of simplicial complexes.

As a last point, we have seen how working at the persistence module level might be favorable for dealing with the persistence of harmonics. In particular, we have shown how to overcome the issues expressed in \cref{rem:harmonics-not functorial} in the tracking of harmonic representatives. That case can benefit from the SVD factorization for real coefficients to lower the computational complexity (\cref{subsec:real}). From the geometrical point of view, we have shown how the interval basis choice for generators applies to the harmonic case. Hence, our work contributes to combining harmonic generators into the persistent homology framework.

\paragraph{Acknowledgements}
The authors acknowledge the SmartData@PoliTO Center for Big Data and Machine Learning support. This study was carried out within the FAIR - Future Artificial Intelligence Research and received funding from the European Union Next-GenerationEU (PIANO \linebreak NAZIONALE DI RIPRESA E RESILIENZA (PNRR) – MISSIONE 4 COMPONENTE 2, INVESTIMENTO 1.3 – D.D. 1555 11/10/2022, PE00000013). 
SS is a member of the Gruppo Nazionale per le Strutture Algebriche, Geometriche e le loro Applicazioni (GNSAGA) of the Istituto Nazionale di Alta Matematica (INdAM) and partially supported by the MIUR Excellence Department Project awarded to the Department of Mathematics, University of Rome Tor Vergata, MatMod@TOV 2023-2027.
This manuscript reflects only the authors’ views and opinions; neither the European Union nor the European Commission can be considered responsible for them. 

%
 \section*{Conflict of interest}

 The authors declare that they have no conflict of interest. 
\newpage
\bibliographystyle{plain} 
\bibliography{references.bib}   

%
%

\newpage
\appendix

\section{Appendix}
\subsection*{Complements to~\Cref{sec:Decomposition}  when $\F=\R$}
\label{subsec:real}

If we use the field $\R$ in the persistence module, we can specialize in decomposing the space described in the previous paragraph. 
We will use the following notation: given a matrix $A$ with $m$ rows and $n$ columns, $A[:,i]$ denotes the $\tth{i}$ column of the matrix, whereas $A[:,:i]$ denotes the submatrix given by the first $i$ columns of $A$. The same notation is used for the first arguments in parenthesis to represent operations on rows. We will make use of this simple result in linear algebra.
\begin{lemma}
\label{lem:restriction-to-orthogonal}
Given three vector spaces $V_1,V_2$, and $V_3$ over $\R$ and two linear maps $\psi_1:V_1\to V_2$ and $\psi_2:V_2\to V_3$ it holds
\begin{equation*}
    \ker( \psi_2\circ \psi_1) = \ker( \psi_1) \oplus \ker \left(\psi_2\circ \psi_1|_{\left(\ker(\psi_1)\right)^\bot}\right).
\end{equation*}
\end{lemma}
\begin{proof}
Let $x$ be an element of $\ker (\psi_2\circ \psi_1)$. It can be written uniquely as $x = v+w$, with $v\in \ker (\psi_1)$ and $w\in (\ker(\psi_1))^\bot$. Since $(\psi_2\circ\psi_1)(v+w)=0$ and $v\in \ker(\psi_1)$, it must be $\psi_2(\psi_1(w))=0$, therefore $w\in \ker (\psi_2\circ\psi_1)$. Then, $w$ belongs to $\ker \left(\psi_2\circ \psi_1|_{(\ker\psi_1)^\bot}\right)$ and the statement follows.
\end{proof}

Fix $M_0$, and suppose that $\varphi_n = 0$. 
For each $M_i$, denote with $d_i$ the number $\dim M_i$. Consider $\varphi_0$ and decompose it via the SVD decomposition in $\varphi_0 = U_0S_0V_0^T$. If $r_0 =\rank \varphi_0$, then $k_0 = d_0 - r_0$ is the dimension of $\ker \varphi_0$. Notice that $S_0$ is a matrix $d_1\times d_0$ with non-zero elements only on the first $r_0$ positions on the main diagonal. Therefore, if $e_i$ is the $\tth{i}$ element of the canonical basis of $\R^{d_0}$, with $r_0 < i\leq d_0$, then $\varphi_0V_0e_i= U_0S_0e_i = 0$. Then, a basis of $\ker \varphi_0$ is given by the vectors $\{V_0e_{r_0+1},\dots,V_0e_{d_0}\}$. The index function $J$ attains the value $1$ on all of them. All such vectors will also be in the kernel of the maps $\varphi_{0,j}$ for all $j > 0$. To avoid repetitions, only the restriction of each $\varphi_{0,j}$ on the orthogonal complement of $\ker \varphi_0$ will be considered. This operation will not change the result because of Lemma \ref{lem:restriction-to-orthogonal}.
To do so, consider the map $\tilde{\varphi}_0 = U_0\tilde{S}_0$, where $\tilde{S}_0 = S_0[:,:r_0]$, given by the first $r_0$ columns of $S_0$. Repeating the same process, we will consider $m_1 = \varphi_1\tilde{\varphi}_0$ instead of $\varphi_{0,2}$. Call $d_1 = d_0 - k_0$. Decompose again $m_1 = U_1S_1V_1^T$ and call $r_1 = \rank m_1$ and $k_1 = d_1 - r_1 = \dim \ker m_1$. Again, a basis of $\ker m_1$ is given by the vectors $V_1e_{r_1+1},\dots, V_1e_{d_1}$. Recall that these vectors are expressed in the basis $\{V_0[:,1],\dots, V_0[:,r_0]\}$ of $\ker\varphi_0^\bot$. To return them in the canonical basis of $M_0$ it is sufficient to consider the matrix $\eta_0$ with $d_0$ rows and $r_0$ columns such that $\eta_0[i,j]$ is equal to 1 if $1\leq i=j \leq r_0$ and 0 otherwise. Then, the vectors in the canonical basis of $M_0$ are $\{V_0\eta_0V_1e_{r_1+1},\dots, V_0\eta_0V_1e_{d_1}\}$. In this case, the index function value for these vectors will be 2. For the general step $j$, consider $m_j = \varphi_j\tilde{m}_{j-1} = U_jS_jV_j^T$. The adapted basis of $M_0$ will be updated with the vectors
\begin{equation}
     V_{0}\eta_{0} \dots V_{j-1}\eta_{j-1} V_je_x,\quad r_j+1\leq x\leq d_j,
\end{equation}
and it will be $J(V_{0}\eta_{0} \dots V_{j-1}\eta_{j-1} V_je_x) = j+1$ for every $r_j+1\leq x\leq d_j$.
Once all the vectors are obtained, as in the general case, it is necessary to complete a basis of $\im(\varphi_{i-1})$ to a basis of $M_0$, introducing the vectors in $\mathcal{V}$ in ascending order given by the function $J$. The resulting vectors will be part of the interval basis. \\
The procedure is encoded in \cref{alg:ssd-real}, which makes use of the matrix decomposition routine \cref{alg:matrix-dec} and specializes \cref{alg:ssd-general} to the case of real coefficients. We denote it by \texttt{ssdR}$(M_i)$. Then, the full decomposition of \cref{alg:total-decomposition} can be specialized to the reals by replacing \texttt{ssd}$(M_i)$ with \texttt{ssdR}$(M_i)$.

\begin{algorithm}[!]
\SetAlgoLined
\textbf{Input}: matrix $A$\;
\KwResult{Restriction of $A$ on the space orthogonal to its kernel with respect to a basis $V$ of the domain, $V$ matrix whose columns are a basis of the domain of $A$, $\dim (\ker A)^\bot$ ,$\dim \ker A$}
$U,S,V = \svd(A)$\;
$nz = \rank S$ \;
$d = $ number of columns of $A$ \;
$dk = d - nz$ \;
$R = US[:,:nz]$\; 
\Return{$R$, $V$, $nz$, $dk$}
 \caption{Matrix decomposition}
 \label{alg:matrix-dec}
\end{algorithm}

\begin{algorithm}
\SetAlgoLined
\textbf{Input}: map $\varphi_{i-1}:M_{i-1}\to M_i$, maps $\{\varphi_j:M_j\to M_{j+1})\}, i\leq j\leq N$ \;
\KwResult{Vectors $\mathcal{V}^i_\text{Birth}$}
$U,S,V \leftarrow \svd(\varphi_{i-1})$ \;
$r \leftarrow \operatorname{rank}(\varphi_{i-1})$\;
$\mathcal{U} \leftarrow U[:,:r]$ basis of the image of $\varphi_{i-1}$\;
$\mathcal{V}^i_\text{Birth} \leftarrow \{\}$; 
$lk \leftarrow 0$\;
$R = \operatorname{Id}:M_i\to M_i$\;
$d \leftarrow \dim M_i$ \; 
$V_{tot} \leftarrow I_{d}$\;
\For{$s=0,\dots,N-i$}{
    $R \leftarrow \varphi_{s+i+1}\cdot R$\;
    \eIf{number of rows of $R = 0$}{
    $k \leftarrow$ number of columns of $R$\;
    $V \leftarrow I_{k}$\;
    $nz \leftarrow 0$\;
    $dk \leftarrow k$\;
    }{
    $R, V, nz, dk \leftarrow \operatorname{dec}(R)$\;
    }
    $V_{temp} \leftarrow I_{d}$, $l \leftarrow \operatorname{ord}(V)$, $V_{temp}[:l,:l] \leftarrow V^t$\;
    $V_{tot} \leftarrow V_{tot}\cdot V_{temp}$\;
    \If{$dk>0$}{
    $\mathcal{T} \leftarrow \textrm{bca}(\mathcal{U}, V_{tot}[:,d-lk-dk :\ d-lk ])$\;
    $\mathcal{U} \leftarrow \mathcal{U} \cup V_{tot}[:,d-lk-dk :\ d-lk ]$\;
    $\mathcal{V}^i_\text{Birth} \leftarrow \mathcal{V}^i_\text{Birth}\cup \mathcal{T}$\;
    $J(t) \leftarrow s+1$ for all $t\in \mathcal{T}$\;
    $lk \leftarrow lk + dk$\;
    }
    \If{$nz = 0$ \textbf{or} $|\mathcal{V}|+r = d$}{\textbf{break}\;}
}
\Return{$\mathcal{V}^i_\text{Birth}$, $J$}
 \caption{single step decomposition on $\R$}
 \label{alg:ssd-real}
\end{algorithm}

\FloatBarrier
\section{Appendix}
\label{sec:graded-modules}
\subsection{Graded modules}
\label{subsec:graded-modules}

This section introduces some notation for graded module presentations and their decomposition into cyclic modules, to link the notion of interval basis to the Smith normal form.
Since it is not obvious to find in the literature the adaptation of classical decomposition results (see \cite{Mines1988AAlgebra}) specialized to the graded case, we include proofs of these results.

In the following, we consider the polynomial ring $\F[x]$ endowed with the standard grading structure defined by the monomial decomposition of polynomials, where the degree $\deg x$ of the indeterminate $x$ is set to $1$.
This way, $\F[x]$ is seen as a direct sum of the $\F$-vector spaces $\F[x]_i$ containing monomials of degree $i$.
Moreover, the property $x^{j}F[x]_i\subseteq \F[x]_{i+j}$ holds.
A graded $\F[x]$-module is an $\F[x]$-module admitting a direct sum decomposition into $\F$-vector spaces, called homogeneous parts of degree $i$, such that the action of $x^j$ over each homogeneous element of degree $i$ gives a homogeneous element of degree $i+j$.
We denote by $\deg v$ the maximum degree of the homogeneous components of $v\in M$.
Clearly, $\F[x]$ can be seen as a graded $\F[x]$-module.

This allows us to make explicit the equivalence of categories $\alpha$ already mentioned in \cref{sec:intro}.

A persistence module $\mathcal{M}$ can be associated with a graded $\F[x]$-module $\alpha(\mathcal{M})$ under a well-known equivalence of categories \cite{corbet2018,Carlsson2005DCG}, in the following way: given $\mathcal{M}$ as above, $\alpha(\mathcal{M})$ is defined as $\bigoplus_{i\in\N}\alpha(M_i):=M_1\oplus M_2\cdots \oplus M_n\oplus M_n\oplus M_n\cdots$. The grading structure is obtained by setting $xv=\varphi_{i}(v)$, for each $i\in [n]$ and $v\in \alpha(M_i)=M_i$ and $xv=v$ for $v\in\alpha(M_j)=M_n$ for $j>n$.
The steps $M_i$ are the homogeneous part of degree $i$ of $\alpha(\mathcal{M})$.

Before proceeding, we introduce the shift notation $\F[x](-d)$ for the graded module $\F[x]$ with standard degrees shifted so that the constant polynomial $1$ has degree $d$. Moreover, we restrict to considering homogeneous homomorphisms with degree zero, i.e., preserving degrees.

\begin{definition}
 Let $M$ be a finitely generated graded $\F[x]$-module.
A {\em presentation} of $M$ is a choice of
\begin{itemize}
    \item a finite system of homogeneous {\em generators} $V=\{v_i\}_{i\in I}$ in $M$;
    \item a finite set of homogeneous equations, called {\em relations (or syzygies)}  $S=\{s_j\}_{j\in J}$ in $M$,
\end{itemize} 
such that the following sequence is exact. 

\begin{equation}
	\xymatrix{
\bigoplus_{j\in J} \F[x](-\deg s_j) \ar[r]^-{\sigma}	
&	\bigoplus_{i\in I} \F[x](-\deg v_i) \ar[r]^-{\epsilon}
&	 M \ar[r]^-{}
&    0,
}
\label{eq:presentation}
\end{equation}

\noindent
where the map  $\epsilon:\bigoplus_{i\in I} \F[x](-\deg v_i) \rightarrow M$ sends the $\tth{i}$-standard generator $e_i$ to  $v_i$,
and $\sigma: \bigoplus_{j\in J} \F[x](-\deg s_j) \rightarrow \bigoplus_{i\in I} \F[x](-\deg v_i)$ 
expresses the equations $s_j$ with respect to the standard basis $\{e_i\}_{i\in I}$.
A presentation of $M$ is said to be {\em minimal} if and only the presentation has a minimal number of generators and relations.
    \label{def:presentation}
\end{definition}

In other words, the module $M$ is obtained as the cokernel of $\sigma$ or $\coker(S)$ where, with a small abuse of notation, matrix $S$ is set to have as column $j$ the coefficients of $s_j$.
In this case, we say that $S$ is a {\em presentation matrix} of $M$, and in the following, we will refer to a pair 
$(\{v_i\} , S)$ as a presentation of $M$.

\begin{definition}[Graded Smith Normal Form]
    A presentation matrix $S$ for some graded $\F[x]$-module $M$ 
    is in {\em graded Smith Normal Form} if and only if 
         each non-zero entry, called a pivot, is the unique non-zero entry in its row and column, and the pivot is equal to $x^p$ for some integer $p\geq 0$.
We will call
$\ones(S)$ the set of row indices in $S$ with pivots equal to $1$.
    \label{def:snf-matrix}
\end{definition}
\noindent

To link a graded Smith Normal Form presentation to an interval decomposition, we set the following notation for the {\em cyclic module generated} in $M$ by a homogeneous element $v$

\begin{equation}
    \F[x]v \subseteq M,
    \label{eq:cyclic-submodules}
\end{equation}

\noindent
and the order of the cyclic submodule is defined as the maximum exponent $p$ such that $x^{p-1}v \neq 0$, possibly infinite.

\begin{theorem}
Let $M$ be a graded $\F[x]$-module and $(\{v_i\}_{i\in I}, S)$ a presentation for $M$ with the notation of~\eqref{eq:presentation}.
The matrix $S$ is in graded Smith Normal Form
if only if
the module $M$ decomposes into cyclic submodules as

\[
	M \cong
    \bigoplus_{m=1}^N \F[x]v_{i_m},
	\]
	
\noindent
with $i_1, \dots, i_N$ the indexing obtained by restricting row indices to $I\setminus \ones(S)$. 
Moreover, if the cyclic submodule $\F[x]v_{i_m}$ is of order $p_{i_m}$, then it is isomorphic to $\F[x](-\deg v_{i_m})/(x^{p_{i_m}})$, otherwise $\F[x]v_{i_m}$ is isomorphic to $\F[x](-\deg v_{i_m})$.
	\label{thm:snf-interval-correspondence}
\end{theorem}
\begin{proof}
We reduce to the case $v_i=e_i$, that is $M$ equal to $\bigslant{F}{\im (\sigma)}$ where $F$ is freely generated by $\{e_i\}_{i\in I}$ since the standard homomorphism $e_i\mapsto v_i$  realizes the isomorphism to $M$. 
Clearly, the elements $v_i$, for $i\in I$, generate $M$ by definition of presentation.
Suppose that the presentation matrix $S$ is in graded Smith Normal Form.
Let $i_1, \dots, i_N$ be the indexing obtained by restricting row indices to $I\setminus \ones(S)$. 

The elements $v_{i_1}, \dots, v_{i_N}$ still generate $M$ since a row index $m\in \ones(S)$ implies that $v_{i_m}$ belongs to the image of $S$.

To prove that the sum is direct, notice that a null combination of $v_{i_1}, \dots, v_{i_N}$ belongs to the image of $S$.
Indeed, since $S$ is in graded Smith Normal Form, $\im(\sigma)$ is freely generated by $x^{p_{i_m}}$ for $m=1, \dots, N$, hence all coefficients are zero.
The order of the cyclic module of $v_{i_m}$ is either $p_{i_m}$ if defined, or infinite otherwise.

On the contrary, assume that $M$ is a direct sum of cyclic modules.
Consider the indices $m$ such that $v_{i_m}$ is an element of finite period $p_{i_m}$.
For each $m$, define the column which is zero for all indexes $j=1,\dots, N$ but in position $m$ where it is $x^{p_{i_m}}$.
By construction, $S$ is in graded Smith Normal Form with respect to generators of $M$. The direct sum defining $M$ implies that there are no other relations to be added to $S$ to obtain a presentation of $M$.

\end{proof}

\begin{corollary}
A Smith Normal Form presentation matrix for $\alpha(\mathcal{M})$ provides an interval basis for $\mathcal{M}$.
\label{prop:interval-basis-syzygies}
\end{corollary}
\begin{proof}
The result follows by recalling that an interval basis directly decomposes $\mathcal{M}$ into interval modules, which, by definition, are analogs of cyclic graded submodules.
Then, it suffices to apply~\Cref{thm:snf-interval-correspondence} to the associated module $\alpha(\mathcal{M})$.
\end{proof}

\subsection{Interval basis via Smith Normal Form}
\label{sec:gradedSmith}

In this section, we propose a method to compute an interval basis based on a suitable reduction of a presentation matrix. It is based on a combination of two technical ingredients: first, the construction of a presentation matrix $S$ for $\alpha(\mathcal{M})$ out of a persistence module $\mathcal{M}=\{  (M_i, \varphi_i)\}_{i=0}^n$. 
 This is done in a way that, to our knowledge, was first explicitly envisaged in a technical passage of \cite{corbet2018}. 
 
Next, we proceed by reducing this presentation matrix into graded Smith Normal Form so that each 
 relation column admits a non-zero entry in correspondence of at most one generator. 
 To do that, we adapt the method from~\cite{skraba+2013x}.

\subsubsection{From a persistence module to its presentation matrix}
\label{sec:from-pers-module-to-matrix}
Our first task consists in defining a matrix $S$ such that  $\coker(S)$ is isomorphic to the graded module $\alpha(\mathcal{M})$ associated with the persistence module $\mathcal{M}$.

For each index $i=0, \dots, n$, fix a basis $\mathcal{B}_i = \{v^i_1,\dots, v^i_{m_i}\}$ of the step $M_i$ and let $\Phi_i$ be the $m_{i+1}\times m_i$ matrix expressing $\varphi_i$ with respect to bases $\mathcal{B}_i$ and $\mathcal{B}_{i+1}$.
Let $m$ be equal to $\sum_{i=0}^n m_i$.
Then, we want to define a presentation for $\alpha(\mathcal{M})$ of the kind

\begin{equation*}
	\xymatrix{
\bigoplus_{j=1}^m \F[x](-\deg s_j) \ar[r]^-{\sigma}	
&	\bigoplus_{i=0}^n \bigoplus_{h=1}^{m_i} \F[x](-h) \ar[r]^-{\epsilon}
&	 \alpha(\mathcal{M}) \ar[r]^-{}
&    0,
}
\label{eq:corbet-presentation}
\end{equation*}

\noindent 
where 
the map $\epsilon$ is defined by $e_{i,h} \mapsto v^i_h$,
and where we want to determine a square matrix $S$ with columns $(s_1,\dots, s_m)$ representing the map $\sigma$.

We follow the construction in Lemma 6 of \cite{corbet2018} and specialize it to modules with no free cyclic submodules. 
Begin by defining $S$ as a matrix of size $m \times m$. We define $S$ by defining some blocks within it. 
We will use the following notation: given a matrix $A$, by $A[:,j]$ we indicate the $\tth{j}$ column of the matrix, 
by $A[i,:]$ we indicate the $\tth{i}$ row of the matrix, whereas $A[{:},{:j}]$ denotes the submatrix given by the first $j$ columns of $A$. The same notation is used for the first arguments in parenthesis to denote operations on rows.
By $A[i:i',j:j']$, we indicate the submatrix given by the rows of  $A$ from $i$ to $i'$ and columns of $A$ from $j$ to $j'$.

Let $d_i := \sum_{j<i}m_j + 1$, for each index $i=0, \dots, n$ (i.e., $d_i$ is the index of the first generator of the $i^{th}$ step).
For each step $i=0,\dots,n$, matrix $S$ contains a $m_i \times m_i$ diagonal block, whose diagonal elements are $-x$:

\[
S[d_i : d_{i+1}-1 , d_i : d_{i+1}-1 ] = -x \ \Id{m_i \times m_i}.
\]

Also for each $i=0, \dots, n$, consider the block $S_i$ below the main diagonal with column indices $d_i, \dots, d_{i+1}-1$ and row indices $d_{i+1}, \dots, d_{i+2}-1$. 
Set 

\[
S[d_{i+1} : d_{i+2}-1 , d_i : d_{i+1}-1] = \Phi_i.
\]

Notice that $S$ is not a diagonal block matrix. This will impact the computational complexity of the reduction procedure.

\begin{definition} \label{def:presentation-matrix}
Given a persistence module $\mathcal{M}$,  the {\em persistence module presentation matrix} is the matrix $S$ obtained as above.
\end{definition}

\begin{theorem}
A persistence module $\mathcal{M}$ and its persistence module presentation matrix $S$ satisfy

\[
\alpha(\mathcal{M}) = \coker(S).
\]

\label{thm:alpha}
\end{theorem}

\begin{proof}
The proof follows from the proof of Lemma 6 in~\cite{corbet2018}.
\end{proof}

\begin{example} \label{ex:gSNF}
\begin{figure}
    \centering
    \includegraphics[scale=0.2,trim={2cm 5cm 1cm 8cm},clip]  {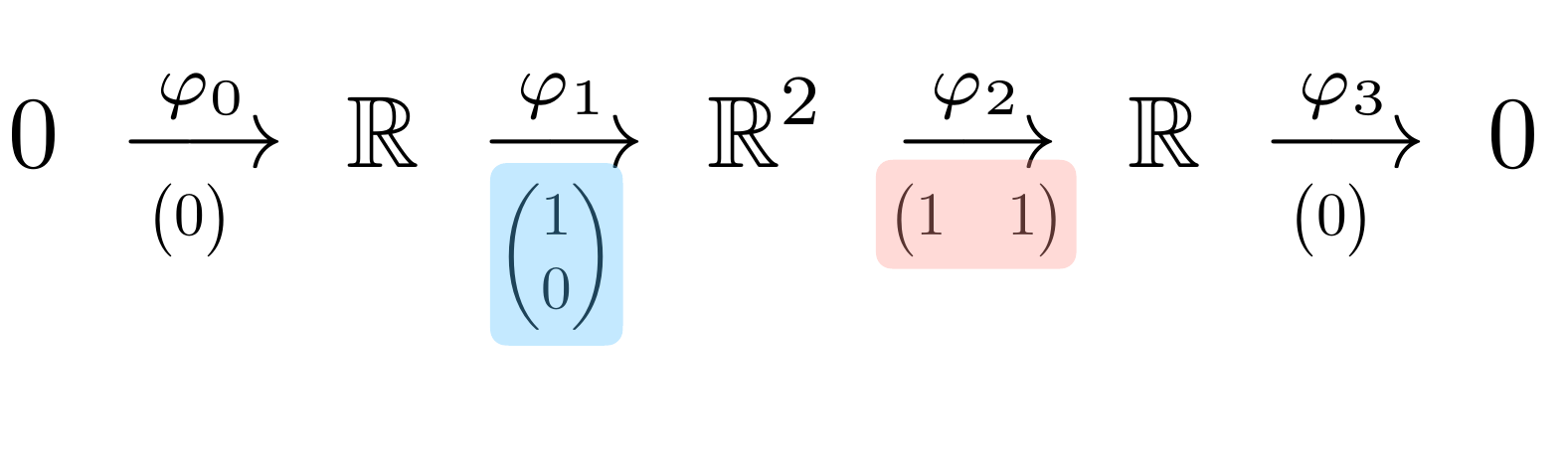}
    \caption{}
    \label{fig:exampleColor}
\end{figure}
\FloatBarrier
Consider the $\R$-persistence module in \cref{fig:exampleColor}, which coincides with the running example in \cref{ex:RunningExPersMod} for $\F = \R$. The matrices below each arrow represent the map above it in the bases $\mathcal{B}_i$'s. Notice that this also corresponds to the $1$-homology persistence module of \cref{fig:example_filtration} for real coefficients.
We ignore the zero steps as they are immaterial to the matrix construction. We say the module has three steps $M_1 = \R$ of degree $1$, $M_2 = \R^2$ of degree $2$ and $M_3 = \R$ of degree $3$. Matrix $S$ is $4 \times 4$, and it holds $d_1=1, \ d_2=2, \ d_3=4$. The matrix $S$ is as in the following \cref{fig:matrixColor}. The ochre blocks are the diagonal blocks, and the cyan and red blocks correspond to matrices $\varphi_1$ and $\varphi_2$ respectively (see colors in \cref{fig:exampleColor}).
\FloatBarrier
\begin{figure}[h]
    \centering
    \includegraphics[scale=0.2,trim={5cm 5cm 5cm 6.8cm},clip]{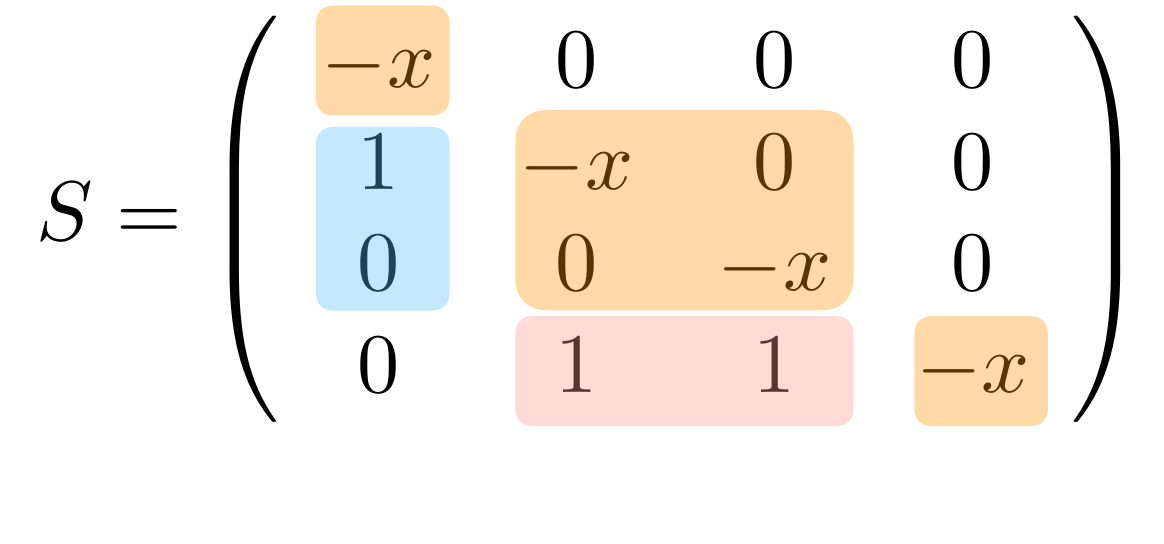}
    \caption{}
    \label{fig:matrixColor}
\end{figure}
\end{example}

Notice that the matrix $S$ represents a homogeneous homomorphism with respect to row and column grades. Hence, matrix $S$ can be considered with entries in $\F$. The only genuinely relevant information in the matrix is whether an element is zero or not because, other than that, its degree is determined by its position. 

\subsubsection{From a presentation matrix to its graded Smith Normal Form}
\label{sec:algo-graded-normal}
In general, the presentation obtained via Definition \ref{def:presentation-matrix} is far from being minimal because several pairs of generator-relation are in excess and can be discarded while maintaining a presentation of the same module. 
As seen in the previous section, we can obtain the interval generators if we find a suitable presentation. This amounts to obtaining a graded version of the structure theorem via the Smith Normal Form, and to the best of our knowledge, has only been explicitly done in \cite{skraba+2013x}.
\begin{theorem}{\cite{skraba+2013x}}
Let $M$ be a finitely-generated, graded $\F[x]$-module, and let $(\{ v_i \}, S)$ be a graded presentation of $M$. An algorithm exists to obtain another presentation of $M$, $(\{ v'_i \}, S')$ such that $S'$ is in graded Smith Normal Form. 
\end{theorem}

We apply the algorithm introduced in~\cite{skraba+2013x} in reducing the square matrix $S$ of size $m\times m$ with entries in $\F$.
The procedure returns invertible $m\times m$ matrices $R, C$ and $\snf(S)$ such that the matrix

\[
\snf(S) := RSC
\]

\noindent
is diagonal up to the reordering of rows and columns, and row and column degrees are preserved.

We sketch the algorithm as follows. By \textit{low} of a column, we refer to the index of its last (downward) non-zero entry. Notice no column of $S$ is zero at the beginning. Also, we disregard matrix $C$, as it is of no interest to us.

\begin{algorithm}
\SetAlgoLined
\textbf{Input}: Matrix $S$ as per Definition \ref{def:presentation-matrix} \;
\KwResult{Matrices $\snf(S)$ and the change of basis matrix $R$}
$R \leftarrow \Id{m\times m}$ \;
\For{$i = 1 , \dots, m$}{
    $l \leftarrow$ low of column $i$ of $S$\;
    $R[l,:]$ $\leftarrow$ $R[l,:] / R[l,i]$\;
    $S[l,:]$ $\leftarrow$ $S[l,:] / S[l,i]$\;
    \For{$j= l-1, \dots, 1$}
    {$R[j,:]$ $\leftarrow$ $R[j,:] - S[j,i]R[l,:]$\;
    $S[j,:]$ $\leftarrow$ $S[j,:] - S[j,i]S[l,:]$\;
    }
    \For{$c=i+1, \dots, m$  }{
    $S[:,c]$ $\leftarrow$ $S[l,i]S[:,c] - S[l,c]S[:,i]$\;
    }
    $\snf(S) \leftarrow S$\;
}
\Return{ $\snf(S) , R $ }
 \caption{Graded Smith Normal Form \cite{skraba+2013x}}
 \label{alg:graded-SNF}
\end{algorithm}

Unlike the non-graded case, swapping rows and columns is not allowed among the typical elementary operations.
This explains the possibly non-diagonal final form in the graded counterpart $\snf(S)$ of the Smith Normal Form of $S$.
For each pivot in $\snf(S)$ corresponding to row $i$ and column $j$, we set $J(i) := \deg j - \deg i$.
If the row $i$ is null, we set $J(i) = \infty$. 

We close this section by linking the matrices $\snf(S)$ and $R$ to the interval basis of the persistence module $\mathcal{M}$ introduced at the beginning of this section.

\begin{theorem}
The columns in $R^{-1}$ of index $m$ such that $J(m)>0$ form an interval basis for $\mathcal{M}$.
\end{theorem}
\begin{proof}
By \cref{thm:alpha}, matrix $S$ defines a presentation of $\alpha(\mathcal{M})$ directly encoding $\mathcal{M}$.
As shown in~\cite{skraba+2013x}, $\snf(S)=RSC$ still defines a presentation of $\alpha(\mathcal{M})$.
Columns in $R^{-1}$ contain the new system of generators with respect to the generators in $S$.
Columns and rows in $\snf(S)$ contain at most one non-zero entry equal to some power of $x$. Hence, the matrix $\snf(S)$ is in graded Smith Normal Form according to Definition \ref{def:snf-matrix}. 
Observe that columns $v_m$ of index $m=1, \dots, N$ such that $J(m)>0$ correspond to non-invertible pivots since there are no null rows in $\snf(S)$.
Hence by Proposition \ref{prop:interval-basis-syzygies}, the set $\{ v_m\}_{m=1}^N$ forms an interval basis for $\alpha(M)$ and each element $v_m$ has associated interval of order $J(m)$.
\end{proof}

\begin{example}{(continued from \ref{ex:gSNF})} \label{eg:snf-reduction}
From matrix $S$, let us compute an interval basis. 
The Smith Normal Form reduction yields 

\begingroup 
\setlength\arraycolsep{4pt}

\[ \snf(S) = \begin{pmatrix} 0 & 0 & 0 & x^3 \\ 
                        1 & 0 & 0 & 0 \\
                        0 & 0  & -x & 0 \\
                        0 & 1  & 0 & 0 \end{pmatrix} \]

We see rows $2$ and $4$ of $SNF(S)$ correspond to surplus generators, as they contain a unit in $\F[x]$. Row $1$ corresponds to a bar born at degree $1$ and killed by a relation (column) of degree $4$, yielding a pair $(1,4)$. Row $3$ corresponds to a bar born at degree $2$ and killed by a relation of degree $3$, yielding a pair $(2,3)$. 
The change of basis matrix $R$ is

\[ R = \begin{pmatrix} -1 & -x & -x & -x^2 \\ 
                        0 & 1 & 0 & 0 \\
                        0 & 0  & 1 & 0 \\
                        0 & 0  & 0 & 1 \end{pmatrix} \]

\noindent
whose inverse equals itself

\[ R^{-1} = \begin{pmatrix} -1 & -x & -x & -x^2 \\ 
                        0 & 1 & 0 & 0 \\
                        0 & 0  & 1 & 0 \\
                        0 & 0  & 0 & 1 \end{pmatrix} \]

\endgroup 
Then, columns $1$ and $3$ in this matrix, corresponding to non-zero persistence generators, form an interval basis. They are $-v^1_1$ and $-xv^1_1 + v^2_2$. They are indeed the first cycle to be born (up to a minus sign, which is immaterial), and the difference between the first cycle mapped at the second step and the second cycle. We remark that $xv^1_1 = v^2_1$. 
Notice that when implemented in practice, the terms of positive degree are substituted by their coefficient, as their degree is implicit by their position.

\end{example}

We have implemented this procedure as Python code, as a purely numerical matrix construction and reduction scheme, and plan to render it publicly available soon. 

\newpage
\section{Appendix}
\subsection*{Complementary Results to \Cref{sec:PersistentHomology}} 
\label{subsec:homology-algorithms}

In this section, we include complementary algorithms to implement the construction of a persistence module obtained through the $\tth{k}$-homology functor applied to a finitely generated chain complex or a chain map between finitely generated chain complexes.
In the following, we describe two methods that can possibly admit an implementation that is distributed over the persistence module steps.

In order to compute the persistent homology $\{ (H^i_k, \tilde{f}_i) \}$, for $i=1,\dots, n$ of a sequence of chain complex maps $\{ f_i \}$, we act in parallel over $i=1,\dots, n$.

\subsubsection*{Computing the homology steps in parallel }

A possible algorithm to retrieve the homology of a chain complex over any coefficient field.

\begin{algorithm}
\SetAlgoLined
\textbf{Input}: Boundary matrices $\partial_k, \partial_{k+1}$ of the chain complex $C$ \;
\KwResult{Betti number $\beta_k$ and basis $\{h_1,\dots, h_{\beta_k}, b_1, \dots, b_r\}$ of $Z_k$, where $\Span{[h_1],\dots, [h_{\beta_k}]} = H_k$ and $\Span{b_1,\dots, b_r}=B_k$. }
Compute the reduction $R_k = \partial_k V_k$ \;
Compute the reduction $R_{k+1} = \partial_{k+1} V_{k+1}$ \;
$b_1,\dots, b_r := $ non-zero columns of $R_{k+1}$ \;
$v_1,\dots,v_s := $ columns of $V_k$ corresponding to zero columns of $R_k$ \;
$J := $ matrix with columns $\{b_1,\dots,b_r, v_1,\dots, v_s\}$ \;
$\beta_k = 1$ \;
\For{$i= r+1, ... r+s$}{ 
    \While{$\exists j<i \ s.t. \ low(J[i]) = low(J[i])$}{
    $l := \text{low}(J[i])$\;
    $\gamma := J[l,i]/J[l,j]$\;
    $J[i] = J[i] - \gamma J[j]$\;
    }
    \If{J[i] is non-zero}{
   $h_{\beta_k}:= J[i]$\;
   $\beta_k = \beta_k + 1$\;
   }{}
}
\Return{ $\beta_k$, basis $\{h_1,\dots, h_{\beta_k}, b_1, \dots, b_r\}$}
 \caption{Computing homology}
 \label{alg:homology}
\end{algorithm}

\newpage
\subsubsection*{Computing the homology structure maps in parallel} 

An algorithm to retrieve the map induced between homology spaces, given a chain map.

\begin{algorithm}
\SetAlgoLined
\textbf{Input}: Chain map $f_k:C_k\to D_k$, representatives cycles $h_1^C, \dots, h_{\beta_k^C}^C$  of a basis of $H_k(C)$, $\beta_k^D$ and $\{h_1^D,\dots, h_{\beta_k^D}^D, b_1^D, \dots, b_r^D\}$ output of \cref{alg:homology} for $D$\;
\KwResult{map $\tilde{f}_k:H_k(C)\to H_k(D)$ induced by $f_k$.}
$\tilde{f}_k := $ zero matrix $\beta_k^D \times \beta_k^C$ \;
\For{$i = 1,\dots,\beta_k^C$}{
    Solve $f_k(h_i^C) = \sum_{j=1}^{\beta_k^D}\lambda_j h_j^D + \sum_{l=1}^r\mu_l b_l$ \;
    $\tilde{f}_k[i] = (\lambda_1,\dots, \lambda_{\beta_k^D})^T$
}
\Return{ $\tilde{f}_k$ }
 \caption{Induced map between homology spaces}
 \label{alg:general-induced-map}
\end{algorithm}

\begin{theorem}
The map $\tilde{f}_k$ defined in \cref{alg:general-induced-map} is well-defined, and it is the map induced by $f_k$ through the homology functor.
\label{thm:induce-hom-correctness}
\end{theorem}

\begin{proof}
For all $i = 1, \dots , \beta_k^C$, it holds $
   \tilde{f}_k([h_i^C]_C)=\left[ f_k(h_i^C)\right]_D = \sum_{j=1}^{\beta_k^D}\lambda_j [h_j^D]_D$.
\end{proof}

\end{document}